\documentclass[a4paper,10pt,reqno]{amsart}
\usepackage{amssymb,latexsym,amsmath,amsfonts}
\usepackage{amsthm}
\usepackage[mathscr]{eucal}
\usepackage{rotating}
\usepackage{multirow}

\numberwithin{equation}{section}
\newtheorem{theorem}{Theorem}
\newtheorem{lemma}{Lemma}

\newtheorem{corollary}{Corollary}

\theoremstyle{definition}
\newtheorem{definition}[theorem]{Definition}

\theoremstyle{remark}
\newtheorem{remark}{Remark}

\newcommand {\sgn}{\mbox{sgn}}

\begin{document}

\title[Polynomial asymptotics of the stochastic pantograph equation]
{Sufficient Conditions for Polynomial Asymptotic Behaviour of the Stochastic Pantograph Equation
}

\author{John A. D. Appleby}
\address{School of Mathematical
Sciences, Dublin City University, Glasnevin, Dublin 9, Ireland}
\email{john.appleby@dcu.ie} 

\author{Evelyn Buckwar}
\address{Johannes Kepler University, Institute for Stochastics, Altenbergerstrasse 69, 4040 Linz, Austria}\email{Email: Evelyn.Buckwar@jku.at}

\thanks{John Appleby gratefully acknowledges the support of a SQuaRE activity entitled ``Stochastic stabilisation of limit-cycle dynamics in ecology and neuroscience'' funded by the American Institute of Mathematics. 
} 
\date{1 July 2016}

\begin{abstract}
This paper studies the asymptotic growth and decay properties of
solutions of the stochastic pantograph equation with
multiplicative noise. We give sufficient conditions on the
parameters for solutions to grow at a polynomial rate in $p$-th
mean and in the almost sure sense. Under stronger conditions the
solutions decay to zero with a polynomial rate in $p$-th mean and
in the almost sure sense. When polynomial bounds cannot be achieved, we show for 
a different set of parameters that exponential growth bounds of solutions in $p$-th mean and an almost sure sense can be obtained. Analogous results are established for
pantograph equations with several delays, and for general finite
dimensional equations.
\end{abstract}

\maketitle

\section{Authors' note}
Much of the contents of the following paper was written in 2003, and a preprint of the work has been available online under this title since that date \cite{0}. To the best of the authors' knowledge in 2003, this work was the first concerning the asymptotic behaviour of stochastic equations with proportional (or indeed unbounded point) delay. Such stochastic proportional delay equations are often called stochastic pantograph equations. The paper did not find a ready home at the time, but subsequently has steadily attracted citations through its online preprint incarnation. These are included in the bibliography below \cite{3,4,1,6,13,8,7,5,9,2,15,10,16,12,11,14}. Since a number of other works quote \cite{0}, we feel it best that the paper be subject to formal review, and as the first author's introduction to the subject came through a paper in the EJQTDE~\cite{MakTer}, we felt it fitting, after a long (but not unbounded) delay, to submit a revised version of it here. 

In fact, asymptotic analysis of stochastic equations with unbounded delay of this type have been the subject of many works, as can be confirmed by investigation of citation databases. Results giving general asymptotic rates, rather than the polynomial behaviour recorded here, have been extended since 2003 to deal with uncertain neural networks, as for example can be seen in \cite{17}. 

We are also grateful for the support of American Institute of Mathematics to enable us to make the appropriate revisions and updates to the work.  

\section{Introduction}
In this paper we shall study the asymptotic behaviour of the
stochastic pantograph equation
\begin{subequations} \label{stochpanto}
\begin{eqnarray}
dX(t)  &= & \{aX(t)+bX(qt)\} \,dt \ + \ \{\sigma X(t)+ \rho
X(qt)\} \,dB(t), \ t \geq 0,
\label{pantogl} \\
X(0)  &= & X_0.   \label{pantoanfang}
\end{eqnarray}
\end{subequations}
We assume that $a, b,\sigma, \rho $ are real constants and $q \in
(0,1)$. The process $(B(t))_{t\geq 0}$ is a standard one
dimensional Wiener process, given on a filtered probability space
$(\Omega, \mathcal{F}, \mathbb{P}) $. The filtration is the
natural filtration of $B$. The initial value $X_0$ satisfies
$\mathbb{E}(|X_0|^2) <\infty$ and is independent of $B$.

We denote a solution of (\ref{stochpanto}), starting at $0$
and with initial condition $X_0$ by $(X(t;0,X_0))_{t\geq 0}$. By virtue of 
\cite{BakerBuck2} there exists a pathwise unique strong solution
$(X(t;0,X_0))_{t\geq 0}$ of (\ref{stochpanto}).

Equation (\ref{stochpanto}) is a generalisation of the
deterministic {\it pantograph equation}
\begin{equation} \label{detpanto}
x'(t) \ = \ {\bar a} x(t) \ + \ {\bar b} x(qt), \ t \geq 0, \quad
x(0)\ = \ x_0, \qquad q\in (0,1),
\end{equation}
in which it is conventional to take $x'(t)$ to denote the
right-hand derivative of $x$.

Since $qt < t$ when $t \geq 0$, equations (\ref{stochpanto}) and
(\ref{detpanto}) are differential equations {\it with time lag}.
The quantity $\tau(t) = t-qt$ in the delayed argument of
$x(t-\tau(t))$ will be called the (variable) lag. We note that the
argument $qt$ satisfies $qt \to \infty$ as $t \to\infty$ but the
lag is unbounded, that is $t-qt \to \infty$ as $t \to\infty$. In
the literature equations like (\ref{stochpanto}) and
(\ref{detpanto}) are also termed (stochastic) delay, retarded or
functional differential equations.

Equation (\ref{detpanto}) and its generalisations possess a wide
range of applications. Equation (\ref{detpanto}) arises, for
example, in the analysis of the dynamics of an overhead current
collection system for an electric locomotive or in the problem of
a one-dimensional wave motion, such as that due to small vertical
displacements of a stretched string under gravity, caused by an
applied force which moves along the string (\cite{Foxetal} and
\cite{OckTay}). Existence, uniqueness and asymptotic properties of
the solution of (\ref{detpanto}) and its generalisations have been
considered in \cite{CarrD,Iser,KatoMcL,MakTer}. Equation
(\ref{detpanto}) can be used as a paradigm for the construction of
numerical schemes for functional differential equations with
unbounded lag, {\it cf.} \cite{Brun}, \cite{Iser0}, \cite{Liu} (we
do not attempt to give a complete list of references here).

A wealth of literature now exists on the non-exponential (general)
rates of decay to equilibrium of solutions of differential and
functional differential equations, both for deterministic and
stochastic equations. Three types of equations which exhibit such
general (non-exponential) rates of decay have attracted much
attention. These are
\begin{itemize}
\item[(i)] Non-autonomous perturbations or forcing terms added to
linear or near-linear problems (such as quasi-linear, or
semi-linear equations).
\item[(ii)] Nonlinear equations (which have no linear, or linearisable
terms near equilibrium), giving rise to weak exponential
asymptotic stability.
\item[(iii)] Certain types of linear equations with unbounded delay.
\end{itemize}
In the deterministic theory, all three mechanisms have been
studied extensively. For stochastic differential equations, and
functional differential equations, several authors have obtained
results in categories (i), (ii), but comparatively few results
have been established in category (iii). We will briefly review
the literature on non-exponential stability of solutions of SDEs
and SFDEs in categories (i), (ii), and allude to the relevant
theory for deterministic problems in category (iii).

An important subclass of non-exponential behaviour is the
so-called polynomial asymptotic stability, where the rate of decay
is bounded by a polynomial with negative exponent,
in either a $p$-th mean or almost sure sense. This type of
stability has been studied in Mao~\cite{Mao2,Mao:92b}, Mao and
Liu~\cite{MaoLiu3}, for stochastic differential equations and
stochastic functional differential equations with bounded delay,
principally for problems of type (i). More general rates of decay
than polynomial are considered in these papers, but in most cases,
it is the non-exponential nature of non-autonomous perturbations,
that gives rise to the non-exponential decay rates of solutions.

The problem (ii) has been investigated for stochastic differential
equations in Zhang and Tsoi~\cite{ZhangTsoi96,ZhangTsoi97}, and
Liu~\cite{LiuK} with state--dependent noise and in Appleby and Mackey~\cite{appmack2003} and Appleby and Patterson~\cite{appleby_patterson} for state independent noise. In these papers, it is principally the nonlinear
form of the equation close to equilibrium that gives rise to the
slow decay of the solution to equilibrium, rather than
non-autonomous time-dependent terms (for deterministic functional
differential equations, results of this type can be found in
Krisztin~\cite{Kris}, and Haddock and Krisztin
~\cite{HaddKris84,HaddKris86}; general decay rates for
deterministic problems of types (i), (ii) are considered in
Caraballo~\cite{Cara}).

The third mechanism (iii) by which SFDEs can approach equilibrium
more slowly than exponentially has been less studied, and
motivates the material in this paper. To this end, we briefly
reprise the convergence properties of linear autonomous
deterministic functional differential equations with bounded delay
and unbounded delay. Bounded delay equations of this type must
converge to zero exponentially fast, if the equilibrium is
uniformly asymptotically stable. However, convergence to
equilibrium need not be at an exponential rate for equations with
{\it unbounded} delay. For example, for a linear convolution
Volterra integro-differential equation, Murakami showed
in~\cite{Mura} that the exponential asymptotic stability of the
zero solution requires a type of exponential decay criterion on
the kernel, Appleby and Reynolds~\cite{jadwr:2} have determined
exact sub-exponential decay rates on the solutions of Volterra
equations, while Kato and McLeod~\cite{KatoMcL} have shown that
solutions of the linear equation (\ref{detpanto}) can converge to
zero at a slower than exponential (polynomial) rate.

In contrast to categories (i), (ii) for SDEs and SFDEs, less is
known regarding the non-exponential asymptotic behaviour of linear
stochastic functional differential equations with unbounded delay,
although it has been shown in~\cite{A1,A2,AR3,ARie1,jadwrs1} that solutions of
linear convolution It\^{o}-Volterra equations can converge to
equilibrium at a non-exponential rate.

In this paper, we show that, in common with the deterministic
pantograph equation studied in~\cite{KatoMcL}, solutions of the
stochastic pantograph equation (\ref{stochpanto}) can be bounded
by polynomials in both a $p$-th mean and almost sure sense, and,
for values of the parameters $a$, $b$, $\sigma$, $\rho$, $q$, we
establish polynomial asymptotic stability in these senses.
Furthermore, it appears, in common with the deterministic
pantograph equation, that the polynomial asymptotic behaviour is
determined only by the values of the parameters associated with
the non-delay terms. We also observe, when the noise intensities
$\sigma$, $\rho$ are small, that the polynomial asymptotic
behaviour of the stochastic problem can be inferred from that of
the corresponding deterministic equation. Our analysis involves
obtaining estimates on the second mean of the solution of
(\ref{stochpanto}) using comparison principle arguments (as
in~\cite{BakerBuck2}), and then using these estimates to
obtain upper bounds on the solution in an almost sure sense, using
an idea of Mao~\cite{Maobk2}.

The article is organised as follows: in Section \ref{prelim} we
give the definitions of the asymptotic behaviour of solutions that
we want to discuss and we state the relevant properties of the
deterministic pantograph equation.

In Section \ref{asymbehav} sufficient conditions are given for
which the solution process is bounded asymptotically by
polynomials, in a first mean and a mean-square sense, as well as
in an almost sure sense. On a restriction of this parameter set,
we show that the equilibrium solution is asymptotically stable in
a $p$-th mean sense $(p = 1, 2)$ or almost surely, with the decay
rate bounded above by a polynomial.

In Section \ref{expupper} we consider unstable solutions of
(\ref{stochpanto}) and parameter regions in which the polynomial
boundedness of these solutions has not been established. We prove
that all such solutions are bounded by increasing exponentials in
the first mean and in mean square and almost surely.

In the penultimate section of the paper, we show that the analysis
of the scalar stochastic pantograph equation with one proportional
delay extends to equations with arbitrarily many proportional
delays, and also to finite dimensional analogues of
(\ref{stochpanto}). The final section discusses some related
problems; an Appendix contains several technical results.

\section{Preliminary results} \label{prelim}

We state the definitions for the asymptotic behaviour in this
section. We follow the definition given in Mao~\cite{Mao3,Mao2}.
\subsection{Definitions of asymptotic behaviour}
First we define the notions of asymptotic growth that we consider
in this article. Notice that 
\begin{equation} \label{eq.Xlininitcondn}
X(t;0,X_0)=X_0 X(t;0,1), \quad t\geq 0.
\end{equation}
Therefore, as $X_0$ is independent of $B$, bounds on the $p$--th moment of $X(t;0,X_0)$ will 
be linear in $\mathbb{E} ( |X_0|^p)$. This fact is reflected in the definition below. 
\begin{definition} \label{growthdef}
Let $(X(t;0,X_0))_{t \geq 0}$ be the unique solution of the SDDE
(\ref{stochpanto}) and $p>0$. We say that the solution is
\begin{subequations}
\begin{enumerate}
 \item {\em globally polynomially bounded in $p$-th mean}, if there exist
constants $C>0$, $\alpha_1\in\mathbb{R}$, such that
  \begin{equation} \label{stabil1}
    \mathbb{E}( |X(t;0,X_0)|^p) \leq C\  \mathbb{E} ( |X_0|^p)\ (1+ t)^{\alpha_1}, \quad t\geq 0;
   \end{equation}
 \item {\em almost surely globally polynomially bounded}, if there exists a
constant $\alpha_2 \in  \mathbb{R}$ such that
   \begin{equation} \label{stabil2}
   \limsup_{t \rightarrow \infty}
 \frac{\log|X(t;0,X_0)|}{\log t}\leq \alpha_2, \quad \mbox{almost
 surely};
    \end{equation}
\item {\em globally exponentially bounded in $p$-th mean}, if there exist
constants $C>0$, $\alpha_3\in\mathbb{R}$, such that
  \begin{equation} \label{stabil1exp}
    \mathbb{E} (|X(t;0,X_0)|^p) \leq C\  \mathbb{E} ( |X_0|^p)\ e^{\alpha_3 t}, \quad t\geq 0;
   \end{equation}
 \item {\em almost surely globally exponentially bounded}, if there exists a
constant $\alpha_4 \in  \mathbb{R}$ such that
   \begin{equation} \label{stabil2exp}
   \limsup_{t \rightarrow \infty}
 \frac{\log|X(t;0,X_0)|}{ t}\leq \alpha_4, \quad \mbox{almost
 surely};
    \end{equation}
\end{enumerate}
\end{subequations}
\end{definition}
In the case that (1) holds for some $\alpha_1\geq 0$, we have that (3) holds for some $\alpha_3\geq 0$, and if (2) holds for some $\alpha_2\geq 0$, then (4) holds for some $\alpha_4\geq 0$.  Also, if there is $\alpha_3\leq 0$ such that (3) holds, then 
there is $\alpha_1\leq 0$ such that (1) holds, and the existence of an $\alpha_4\leq 0$ such that (4) holds implies the existence of an $\alpha_2\leq 0$ such that (2) holds. 

We stop to justify and clarify some aspects of our definition. 

First, we note that the constants $\alpha_{i}$ for $i=1,\ldots,4$ in each of (1)--(4) are independent of $X_0$ by dint of \eqref{eq.Xlininitcondn}. 

Second, although we choose to consider bounds in \eqref{stabil1} and \eqref{stabil1exp} for all $t\geq 0$, it is equivalent to start with a definition on which such bounds
hold for $t\geq T$, for some $T>0$.  

To prove this second remark, we concentrate on the polynomial case, noting that the situation is similar in the exponential one. Clearly, if we take as our starting point in lieu of \eqref{stabil1} a polynomial estimate that holds for all $t\geq T$ (where $T$ is sufficiently large), we may assume  
\[
\mathbb{E}[|X(t;0,1)|^p] \leq C'(1+t)^\alpha, \quad t\geq T,
\] 
where $C'>0$ and $\alpha$ are independent of $X_0$. For $t\leq T$ we have 
\[
\mathbb{E}[|X(t;0,1)|^p]\leq \sup_{t\in [0,T]}\mathbb{E}(|X(t;0,1)|^p). 
\]
Now define  
\[
C=\max\left(C',\frac{\sup_{t\in [0,T]}\mathbb{E}(|X(t;0,1)|^p)}{\min(1,(1+T)^\alpha)}\right).
\] 
Again, we see that $C$ is independent of $X_0$. For 
$t\leq T$ we have $\mathbb{E}[|X(t;0,X_0)|^p]=\mathbb{E}[|X_0|^p] \mathbb{E}(|X(t;0,1)|^p)$ and so
\begin{align*}
\mathbb{E}[|X(t;0,X_0)|^p]&\leq 
\mathbb{E}[|X_0|^p](1+t)^\alpha \cdot \frac{\sup_{t\in [0,T]}\mathbb{E}(|X(t;0,1)|^p)}{(1+t)^\alpha} \\
&\leq \mathbb{E}[|X_0|^p](1+t)^\alpha \cdot \frac{\sup_{t\in [0,T]}\mathbb{E}(|X(t;0,1)|^p)}{\min((1+T)^\alpha,1)} \\
&\leq C\mathbb{E}[|X_0|^p](1+t)^\alpha. 
\end{align*}
On the other hand, for $t\leq T$ we have 
\[
\mathbb{E}[|X(t;0,X_0)|^p] = \mathbb{E}[|X_0|^p] \mathbb{E}(|X(t;0,1)|^p) \leq C'\mathbb{E}[|X_0|^p](1+t)^\alpha\leq C\mathbb{E}[|X_0|^p](1+t)^\alpha. 
\]
Hence the estimate $\mathbb{E}[|X(t;0,X_0)|^p] \leq  C\mathbb{E}[|X_0|^p](1+t)^\alpha$ holds for all $t\geq 0$, and this legitimises 
our definition in \eqref{stabil1}.  

We are concerned with the following notions of asymptotic stability.
\begin{definition} \label{stabildef} The equilibrium solution of the SDDE (\ref{stochpanto}) is
said to be
\begin{subequations}
\begin{enumerate}
  \item {\em globally polynomially stable in $p$-th mean}, if (\ref{stabil1}) holds with
 $\alpha_1<0$;
\item {\em almost surely globally polynomially stable}, if (\ref{stabil2})
holds with $\alpha_2<0$.
\end{enumerate}
\end{subequations}
\end{definition}
\begin{remark}
If, for every $\epsilon>0$, the solution of the SDDE
(\ref{stochpanto}) satisfies
 \begin{equation} \label{stabil2a}
   \limsup_{t \rightarrow \infty}
 \frac{|X(t;0,X_0)|}{t^{\alpha_2+ \epsilon}}=0, \quad \mbox{almost
 surely},
    \end{equation}
then the solution is almost surely globally polynomially bounded
with constant $\alpha_2$, i.e. (\ref{stabil2}) holds. Moreover, \eqref{stabil2a} implies \eqref{stabil2}. 
The term {\em global} in the above refers to the fact that we do not require a
restriction on the initial value.
\end{remark}

\subsection{Results for the deterministic pantograph
equation}  First, we provide a comparison result for solutions of
pantograph equations. We mention that stochastic comparison
arguments are used in the study of stochastic delay differential
equations with discrete time lag and their Euler-Maruyama
approximations \cite{BakerBuck3}, and for It\^{o}-Volterra
equations~\cite{1}.

Recall that, for a continuous real-valued function $f$ of a
real variable, the Dini-derivative $D^+ f$ is defined as
\[
D^+ f(t) = \limsup_{\delta \downarrow 0}
\frac{f(t+\delta)-f(t)}{\delta}.
\]
\begin{lemma}  \label{lemma:lemp1}
Let ${\bar b}>0$, $q\in(0,1)$. Assume $x$ satisfies
\begin{equation}   \label{detpantoeps}
x'(t) = {\bar a} x(t) \ + \ {\bar b}x(qt), \ t \geq 0,
\end{equation}
where $x(0)>0$ and suppose $t\mapsto p(t)$ is a continuous
non-negative function defined on $\mathbb{R}^{+}$ satisfying
\begin{equation}    \label{eq:eqg1}
D^+ p(t)\leq {\bar a} p(t) + {\bar b} p(qt), \quad t\geq 0
\end{equation}
with $0<p(0)<x(0)$. Then $p(t)\leq x(t)$ for all $t \geq 0$.
\end{lemma}
\begin{proof} See the Appendix.
\end{proof}
The following result concerning asymptotic properties of the
deterministic pantograph equation has been proved in
\cite{KatoMcL}. The first half of the result (given in
\cite[Section 4, Theorem 3]{KatoMcL}) will be of utility in obtaining
polynomial upper bounds on the $p^{th}$-mean of the process; and
those bounds are in turn used to obtain estimates on the almost
sure asymptotic behaviour of the solution of (\ref{stochpanto}).
The second part will be employed to establish exponential upper
bounds on the solutions of (\ref{stochpanto}); it can be found in
\cite[Section 5, Theorem 5]{KatoMcL}. The case when $\bar{a}=0$ is covered in 
\cite[Section 6, Theorem 7]{KatoMcL}; we have slightly modified notation for convenience. We recall for the third part that for 
$z\in\mathbb{C}$ that the principal value of the logarithm is defined by 
\[
\text{Log}(z):=\log|z|+i\text{Arg}(z)
\]
where $\text{Arg}(z)\in (-\pi,\pi]$ is the principal value of the argument of $z$. 
In particular, for $x<0$, $\text{Log}(x)=\log|x|+i\pi$.
\begin{lemma}    \label{lemma:lemp2}
Let $x$ be the solution of (\ref{detpanto}).
\begin{itemize}
\item[(i)]
If ${\bar a}<0$, there exists $C_1>0$ such that
\[
\limsup_{t \rightarrow \infty} \frac{|x(t)|}{t^{\gamma}}=C_1|x(0)|
\]
where $\gamma\in\mathbb{R}$ obeys
\begin{equation}   \label{eq:eqg3}
0={\bar a} +|{\bar b}| q^{\gamma}.
\end{equation}
Therefore, for some $C>0$, we have
\begin{equation}   \label{eq:eqg4}
|x(t)| \leq C |x(0)| (1+ t)^{\gamma}, \quad t\geq 0.
\end{equation}
\item[(ii)]
If ${\bar a}>0$, there exists $C>0$ such that
\begin{equation}   \label{eq:eqg4pr}
|x(t)|\leq C|x(0)| e^{{\bar a}t}, \quad t\geq 0.
\end{equation}
\item[(iii)] If $\bar{a}=0$, define $c' = \textup{Log}(1/q) > 0$ and set
\[
\psi(t):= t^k (\textup{Log}(t))^h
 \exp\left(\frac{1}{2c'}\left\{\textup{Log}(t) - \textup{Log} \textup{Log}(t)\right\}^2\right),
\]
where
\[
k = \frac{1}{2} +\frac{1}{c'} + \frac{1}{c'}\textup{Log}(\bar{b}c'), 
\quad 
h = - 1 - \frac{1}{c'}\textup{Log}(\bar{b}c').
\]
Then $x(t)=\mathcal{O}(\psi(t))$ as $t\to \infty$, and if $x(0)\neq 0$,
then $x(t)$ is not $o(\psi(t))$ as $t\to \infty$. 
\end{itemize}
\end{lemma}
We can consider $x(0)\neq 0$ in Lemma~\ref{lemma:lemp2}, because if $x(0)=0$, then 
$x(t)=0$ for all $t\geq 0$ and all estimates follow trivially.

The constants $C$ and $C_1$ in Lemma~\ref{lemma:lemp2} are independent of $x(0)$, and indeed the estimates \eqref{eq:eqg3}, 
\eqref{eq:eqg4} and \eqref{eq:eqg4pr} hold for all $t\geq 0$, rather than merely for sufficiently large $t$, as might readily be supposed. This is because the function $x_1$, which is the unique continuous solution of  \eqref{detpanto} with initial condition 
$x_1(0)=1$, can be used to express the solution of \eqref{detpanto} with general initial condition $x(0)\neq 0$. 
Indeed, we have that $x(t)=x(0)x_1(t)$ for all $t\geq 0$. Therefore, applying part (i) of Lemma~\ref{lemma:lemp2} to $x_1$, we see that there exists a constant $C_1=C_1(\bar{a},\bar{b},q)>0$
\[
\limsup_{t \rightarrow \infty} \frac{|x_1(t)|}{t^{\gamma}}=C_1
\]
where $\gamma$ obeys \eqref{eq:eqg3}. Therefore, there must also exist a constant $C=C(\bar{a},\bar{b},q)$ such that 
$|x_1(t)| \leq C_1 (1+t)^{\gamma}$ for $t>0$ from which \eqref{eq:eqg4} immediately follows. An analogous argument applies to part (ii).

We do not discuss in this paper the situation when $\bar{a}=0$ which is covered by part (iii). It can be seen from part (iii) of Lemma \ref{lemma:lemp2} that 
\[
\log|\psi(t)|\sim \frac{1}{2}\frac{1}{\log(1/q)} \log^2(t), \quad \text{as $t\to\infty$}
\] 
so therefore as $x=\mathcal{O}(\psi)$ and $x$ is not $o(\psi)$ we have
\[
\limsup_{t\to\infty} \frac{\log|x(t)|}{ \log^2(t)}=\frac{1}{2}\frac{1}{\log(1/q)}.
\]
Therefore, $x$ enjoys an upper bound which is neither polynomial nor exponential, so we would not expect such bounds to transfer to the corresponding stochastic equation. For this reason such problems are beyond the immediate scope of the paper. In the case $a=0$, $b>0$, $\rho=0$, however, we can use the methods herein to show that $|X(t)|>0$ 
for all $t\geq 0$, and that $m(t)=\mathbb{E}[|X(t)|]/\mathbb{E}[|X_0|$ solves the 
differential equation $m'(t)=bm(qt)$ for $t\geq 0$ with $m(0)=1$. Then the estimate in part (iii) of Lemma~\ref{lemma:lemp2} can be applied to $m$ and thence to $t\mapsto \mathbb{E}[|X(t)|]$. The pure delay stochastic pantograph equation (i.e., with $a=\sigma=0$) 
\[
dX(t)=bX(qt)\,dt + \rho X(qt)\,dB(t)
\]  
is more likely to inherit the type of asymptotic behaviour in part (iii) of Lemma~\ref{lemma:lemp2}, but such a conjecture is the subject of further investigation. 

Now, taking as given that the case $\bar{a}=0$ is excluded henceforth from our discussions, we may combine the results of the previous two lemmas to obtain
the following explicit estimates on the asymptotic behaviour of
continuous functions obeying inequality (\ref{eq:eqg1}).
\begin{lemma}    \label{lemma:lemp3}
Suppose $p(0)>0$ and $t\mapsto p(t)$ is a continuous non-negative
function defined on $\mathbb{R}^{+}$ satisfying (\ref{eq:eqg1})
where ${\bar b}>0$ and $q\in(0,1)$.
\begin{itemize}
\item[(i)] If $\bar{a}<0$, there exists $C>0$ such that
\begin{equation}   \label{eq:eqcompp1}
p(t) \leq C p(0)(1+ t)^{\gamma}, \quad t\geq 0
\end{equation}
where $\gamma$ obeys (\ref{eq:eqg3}).
\item[(ii)]
If $\bar{a}>0$, there exists $C>0$ such that
\begin{equation}     \label{eq:eqcompp2}
p(t) \leq C p(0) e^{\bar{a}t}, \quad t\geq 0.
\end{equation}
\end{itemize}
\end{lemma}
\begin{proof}
We establish part (i) only; the proof of part (ii) follows
similarly. Let $\varepsilon>0$ and $x$ be a solution of
(\ref{detpanto}) with $x(0)=(1+\varepsilon)p(0)>p(0)$. By
Lemma~\ref{lemma:lemp1}, $p(t)\leq x(t)$. By
Lemma~\ref{lemma:lemp2} (i), there is $C'>0$ such that $x(t)\leq
C'x(0)t^{\gamma}$, for $t\geq 0$ where $\gamma$ is given by
(\ref{eq:eqg3}). But then
\[
p(t)\leq x(t)\leq C'x(0)t^{\gamma} = Cp(0)(1+t)^{\gamma}
\]
where $C=C'(1+\varepsilon)$.
\end{proof}

\section{Polynomial asymptotic behaviour} \label{asymbehav}

In this section, we concentrate on giving sufficient conditions
under which the asymptotics of the process satisfying
(\ref{stochpanto}) are polynomially bounded or stable, in both a
$p$-th mean $(p = 1, 2)$ and almost sure sense. The proofs of the
$p$-th mean polynomial asymptotic behaviour follow essentially by
taking expectations across the process $|X(t)|^{p}$, whose
semimartingale decomposition is known by It\^{o}'s rule. This
yields a functional differential inequality whose solution can be
majorised by the solution of a deterministic pantograph equation.

The determination of almost sure polynomial boundedness or
stability of the solution of the stochastic pantograph equation
uses the ideas of Theorem 4.3.1 in Mao, \cite{Maobk2}. There the
idea is to determine a.s. exponential asymptotic stability of
solutions of stochastic functional differential equations with
bounded delay, once a decreasing exponential upper bound is
obtained for a $p$-th moment. These ideas have been modified to
obtain results on almost sure asymptotic stability of
It\^{o}-Volterra equations, using the $p$-th mean integrability of
the solution \cite{JApas}.

\subsection{Polynomial growth bounds in the first and second mean}
In Theorem \ref{theorem:pg1} we give sufficient conditions on the
parameters $a,b, \sigma$ and $\rho$, such that solutions of
(\ref{stochpanto}) are polynomially bounded in the first mean and
in mean-square. We treat the cases $\rho =0$  and $\rho \neq 0$
separately, because in the first case we can obtain a larger set
of parameter values for which the $p$-th moment of the solution is
bounded by a polynomial, for $p\leq 1$. Another benefit of a
separate analysis for the first mean in the case $\rho=0$, is that
we obtain a larger parameter region for which the process is a.s.
polynomial bounded, and a.s. polynomial stable, than if we
initially obtain a bound on the second mean.

We consider polynomial stability of the solutions of
(\ref{stochpanto}) in Theorem \ref{theorem:pg2}.

\begin{theorem}   \label{theorem:pg1}
Let $(X(t))_{t \geq 0}$ be the unique process satisfying
(\ref{stochpanto}).
\begin{itemize}
\item[(i)]
Let $\rho=0$, and $\mathbb{E}[|X_0|^2]<\infty$. 
If $a<0$, there exists a real constant $\alpha$, and
a positive constant $C$ such that
\[
\mathbb{E}( |X(t)| ) \leq C\ \mathbb{E} ( |X_0|)\ (1+ t)^{\alpha}
\qquad \mbox{for} \quad t \geq 0,
\]
where $\alpha$ is given by
\begin{equation}     \label{eq:eqgt1}
\alpha=\frac{1}{\log q}\log\left(\frac{-a}{|b|}\right).
\end{equation}
\item[(ii)]
Let $\rho\not=0$, and $\mathbb{E}[|X_0|^4]<\infty$. 
If $\ 2a+\sigma^{2}<0$, there exists a real
constant $\alpha$, and a positive constant $C$ such that
\[
\mathbb{E}(|X(t)|^{2})\leq C \  \mathbb{E} ( |X_0|^2)\ (1+ t)^{\alpha}
\qquad \mbox{for} \quad t \geq 0,
\]
where $\alpha$ is given by
\begin{equation}    \label{eq:eqgt2pr}
\alpha= \frac{2}{\log q} \log \left( \frac{1}{\rho^{2}} \left(
\sqrt{|b+\sigma\rho|^2-\rho^2(2a+\sigma^2)}-|b+\sigma\rho| \right)
\right).
\end{equation}
\end{itemize}
\end{theorem}
\begin{remark} \label{rem.1}
In Theorem \ref{theorem:pg1}, Part  (i), the exponent
$\alpha$ is sharp for $b>0$, as
\[
\lim_{t \rightarrow \infty} \frac{\mathbb{E}(|X(t)|)}{t^{\alpha}}
= C\ \mathbb{E}(|X_0|).
\]
\end{remark}
\begin{proof}[Proof of Remark~\ref{rem.1}]
To see this, recall that $X(t;0,X_0)=X(t;0,1)X_0$ for $t\geq 0$. Define $Y(t)=X(t;0,1)$ for $t\geq 0$. Then $b>0$ and $\rho=0$, $Y$ obeys
\begin{equation} \label{eq.sdeY}
dY(t)= (aY(t)+bY(qt))\,dt + \sigma Y(t)\,dB(t), \quad t\geq 0; \quad Y(0)=1.
\end{equation}
Then it can be shown, because $Y(0)>0$, that $Y(t)>0$ for all $t\geq 0$ almost surely. This is done by defining $\phi(t)=e^{(a-\sigma^2/2)t+\sigma B(t)}$ for $t\geq 0$. Then $\phi$ obeys the stochastic differential equation 
\[
d\phi(t)= a\phi(t)\,dt + \sigma \phi(t)\,dB(t), \quad t\geq 0; \quad \phi(0)=1.
\] 
Then, using stochastic integration by parts, we deduce that $Q(t):=Y(t)/\phi(t)$ obeys 
\begin{align*}
dQ(t)&=d(Y(t)/\phi(t))=\frac{1}{\phi(t)}dY(t)+Y(t)d\left(\phi(t)^{-1}\right) -\sigma^2 Y(t)\phi(t)\,dt \\
&=b\frac{1}{\phi(t)}Y(qt)\,dt = b\frac{1}{\phi(t)}\phi(qt)Q(qt)\,dt.
\end{align*}
Hence with $\mu(t):=b\phi(qt)/\phi(t)$ for $t\geq 0$, we see that $Q$ obeys the functional differential equation 
\[
Q'(t)=\mu(t)Q(qt), \quad t\geq 0;\quad Q(0)=1.
\]
Since $\phi(t)>0$ for all $t\geq 0$, it follows that $\mu(t)>0$ for all $t\geq 0$. Since $Q(0)=1$, we see that $Q(t)>0$ 
for all $t\geq 0$, using a standard deterministic argument, it follows that $Y(t)=Q(t)\phi(t)>0$ for all $t\geq 0$, as claimed. 

Next, by taking expectations across \eqref{eq.sdeY} leads to 
\[
m(t):= \mathbb{E}[Y(t)] = m(0)+\int_0^t (am(s)+bm(qs))\,ds, \quad t\geq 0.
\]
Since $m$ is continuous, we get 
\[
m'(t)=am(t)+bm(qt), \quad t\geq 0;\quad m(0)=1.
\]
Since $a<0$ and $b>0$, it follows from part (i) of Lemma \ref{lemma:lemp2} that 
\[
\lim_{t\to\infty} \frac{m(t)}{t^{\alpha}}=C_1,
\]
where $\alpha$ obeys \eqref{eq:eqgt1}. Finally, because $Y(t)>0$ for all $t\geq 0$, and $Y$ and $X_0$ are independent we have
\[
m(t)\mathbb{E}[|X_0|]=\mathbb{E}[Y(t)]\mathbb{E}[|X_0|]=\mathbb{E}[|Y(t)|]\mathbb{E}[|X_0|]=\mathbb{E}[|Y(t)X_0|]=\mathbb{E}[|X(t;0,X_0)|].
\]
 Therefore
\[
\lim_{t\to\infty} \frac{\mathbb{E}[|X(t;0,X_0)|]}{t^\alpha}
=\lim_{t\to\infty} \frac{m(t)\mathbb{E}[|X_0|]}{t^\alpha} = C_1 \mathbb{E}[|X_0|],
\]
as claimed. 
\end{proof}
\begin{proof}[Proof of Theorem~\ref{theorem:pg1}] 
Part (i): Notice by (\ref{stochpanto}) that $X$ is a continuous
semimartingale and therefore there exists a semimartingale local
time $\Lambda$ for $X$. By the Tanaka-Meyer formula
\cite[Chap.3,(7.9)]{KarShr} we therefore have
\begin{eqnarray}
|X(t)| &=& |X(0)| + \int_{0}^{t} \sgn(X(s))(aX(s)+bX(qs))\,ds
\nonumber \\
\label{eq:absx} & & \qquad+ \int_{0}^{t}\sgn
(X(s))\sigma X(s)\,dB(s)+2\Lambda_{t}(0), \quad \mbox{a.s.}
\end{eqnarray}
where $\Lambda_{t}(0)$ is the local time of $X$ at the origin. In
fact by Lemma \ref{lemma:LeGall} and the remark following it, we
have
\begin{equation}
\label{loc0} \Lambda_{t}(0) = 0 \quad \mbox{for all } t \geq 0,
\mbox{ a.s.}
\end{equation}
Thus, for any $t\geq0$, $t+h\geq 0$, (\ref{eq:absx}) gives
\begin{eqnarray*}
|X(t+h)|-|X(t)|  &=& \int_{t}^{t+h} \big(a|X(s)|
+\sgn (X(s))\ bX(qs)\big)\,ds \\
& & \qquad \qquad +\int_{t}^{t+h} \sigma |X(s)|\,dB(s) \\
&\leq& \int_{t}^{t+h} \{a|X(s)| + |b||X(qs)|\}\,ds + \int_{t}^{t+h}
\sigma |X(s)|\,dB(s).
\end{eqnarray*}
By Lemma \ref{lemma:lemg2} (ii), we get, with
$m(t)=\mathbb{E}(|X(t)|)$
\[
m(t+h)-m(t)\leq \int_{t}^{t+h} \{am(s)+|b|m(qs)\}\,ds.
\]
By Lemma \ref{lemma:lemg2} (i), $t \mapsto m(t)$ is continuous, so
\begin{equation}   \label{eq:eqcomp3}
D^+ m(t) \leq am(t)+|b|\ m(qt).
\end{equation}
Note that
$m(0)=\mathbb{E}(|X(0)|)\leq(\mathbb{E}(|X(0)|^2))^{1/2}<\infty$.
Since $\alpha$ defined by (\ref{eq:eqgt1}) satisfies
$a+|b|q^{\alpha}=0$, Lemma~\ref{lemma:lemp3} implies there exists
$C>0$ such that
\[
\mathbb{E}(|X(t)|) = m(t) \leq C\;m(0)(1+t)^{\alpha} = C\
\mathbb{E}(|X(0)|)\;(1+t)^{\alpha}, \quad t\geq 0,
\]
as required.

As for the proof of part (ii), in which $\rho\not=0$, we observe
that if $X(t)$ is governed by (\ref{stochpanto}), then
$Y(t)=X^{2}(t)$ is an It\^{o} process, with semimartingale
decomposition given by It\^{o}'s rule applied to $X^{2}(t)$ of:
\begin{multline*}
Y(t)=Y(0)+\int_{0}^{t} \{2a Y(s)+ 2bX(s)X(qs)
+ (\sigma X(s)+\rho X(qs))^{2}\}\,ds \\
+ \int_{0}^{t} \{2\sigma Y(s)+ 2 \rho X(s)X(qs)\}\,dB(s),
\end{multline*}
so for $t\geq 0$, $t+h\geq0$, we have
\begin{multline}   \label{eq:eqg6}
Y(t+h)- Y(t) = \int_{t}^{t+h} \{2 aY(s) + 2b X(s)X(qs)
+ ( \sigma X(s) + \rho X(qs))^2\} \,ds \\
+ \int_{t}^{t+h} \{2\sigma Y(s)+2\rho X(s)X(qs)\}\,dB(s).
\end{multline}
Noting that
\begin{align*}
(2\sigma Y(s)+ 2\rho X(s)X(qs))^{2}
&\leq 2(4\sigma^{2} Y^{2}(s)+4\rho^{2}Y(s)Y(qs)) \\
&\leq 8\sigma X^{4}(s)+4 \rho^{2}(X^{4}(s)+X^{4}(qs)),
\end{align*}
and using the fact (Lemma \ref{lemma:lemg2} (ii)) that
\[
\int_{0}^{t} \mathbb{E}(|X(s)|^{4})\,ds< \infty, \qquad \mbox{for
all} \ t\geq 0,
\]
we have
\begin{equation}   \label{eq:eqg7}
\mathbb{E} \int_{t}^{t+h} (2\sigma Y(s)+ 2\rho
X(s)X(qs))^{2}\,dB(s) = 0.
\end{equation}
Let $\eta$ be any positive constant. Using Young's inequality in
the form $2xy \leq \eta^2 x^{2} + \eta^{-2}y^{2}$ on the right
hand side of (\ref{eq:eqg6}) yields
\begin{multline*}
Y(t+h)- Y(t) \leq \int_{t}^{t+h} \Big\{(2a+\sigma^2)Y(s)
+\rho^2Y(qs) +
\\ |b+\sigma\rho|\left(\eta^{2}Y(s)+\frac{1}{\eta^2}Y(qs)\right)\Big\}\,ds
+ \int_{t}^{t+h} \{2 \sigma Y(s)+ 2\rho X(s)X(qs)\} \,dB(s).
\end{multline*}
Let $\widetilde{m}(t)=\mathbb{E}(Y(t))$. Note from Lemma
\ref{lemma:lemg2} (i) that $t\mapsto \widetilde{m}(t)$ is a
continuous and (by construction) nonnegative function. Taking
expectations both sides of the last inequality, and using
(\ref{eq:eqg7}), we therefore get
\begin{multline*}
\widetilde{m}(t+h) -\widetilde{m}(t) \leq \int_{t}^{t+h} (2a
+\sigma^2 + \eta^2|b+\sigma\rho|)\
\widetilde{m}(s)\,ds \\
+\int_t^{t+h}\left(\rho^{2} + \frac{1}{\eta^2}|b+\sigma\rho| \right)
\widetilde{m}(qs)\,ds.
\end{multline*}
Using the continuity of $\widetilde{m}$, we now get
\begin{equation}   \label{eq:eqcomp4}
D^{+}\widetilde{m}(t) \leq (2a +\sigma^2 +
\eta^{2}|b+\sigma\rho|)\ \widetilde{m}(t) +
\left(\frac{1}{\eta^{2}}|b+\sigma\rho| +\rho^2 \right)
\widetilde{m}(qt).
\end{equation}
Since $2a +\sigma^2 < 0$, we can choose $\eta>0$ such that ${\bar
a} = 2a +\sigma^2 + |b+\sigma\rho|\eta^{2}<0$. Now, using part (i)
of Lemma~\ref{lemma:lemp3}, there exists $C>0$ and
$\alpha\in\mathbb{R}$ such that
\[
\mathbb{E}(|X(t)|^2) =\widetilde{m}(t) \leq
C\;\widetilde{m}(0)\;(1+t)^{\alpha} =
C\;\mathbb{E}(|X(0)|^2)\;(1+t)^{\alpha}, \quad t\geq 0.
\]

Finally we choose $\eta>0$ optimally; i.e., in such a way that
$\alpha=\alpha(\eta)$ given by
\begin{equation}       \label{eq:eqgt2}
2a+\sigma^2+\eta^2|b+\sigma\rho| +
\left(\rho^2+\frac{1}{\eta^2}|b+\sigma\rho|\right)q^{\alpha} =0
\end{equation}
is minimised (such a minimum exists, once $2a+\sigma^2<0$).
The optimal choice of $\eta$ gives rise to the value of $\alpha$
given in (\ref{eq:eqgt2pr}). This value of $\alpha$ gives the
sharpest bound on the growth rate of $\mathbb{E}(|X(t)|^{2})$, in the sense that any other choice of $\eta$ gives rise to a larger 
$\alpha$ than that quoted in  (\ref{eq:eqgt2pr}). The
proof is thus complete.
\end{proof}
The line of reasoning used in Theorem \ref{theorem:pg1}, taken in
conjunction with Lemma \ref{lemma:lemp3}, enables us to obtain
polynomial stability of the equilibrium solution of
(\ref{stochpanto}) in the first and second mean in the sense of
Definition \ref{stabildef} in certain parameter regions.
\begin{theorem}   \label{theorem:pg2}
Let $(X(t))_{t \geq 0}$ be the process uniquely defined by
(\ref{stochpanto}).
\begin{itemize}
\item[(i)]
Let $\rho=0$, and $\mathbb{E}[|X_0|^2]<\infty$. 
If $a+|b|<0$, there exists $C>0$ and $\alpha<0$ such that
\[
\mathbb{E}(|X(t)|) \leq C\;\mathbb{E}(|X_0|)\;(1+ t)^{\alpha}, \quad t
\geq 0,
\]
with $\alpha$ given by (\ref{eq:eqgt1}).
\item[(ii)]
Let $\rho\not=0$, and $\mathbb{E}[|X_0|^4]<\infty$. 
If $2a+\sigma^2+\rho^2+2|b+\sigma\rho|<0$, there exists 
$C>0$, $\alpha<0$ such that
\[
\mathbb{E}(|X(t)|^2) \leq C\;\mathbb{E}(|X_0|^2)\; (1+t)^{\alpha},
\quad t \geq 0,
\]
with $\alpha$ given by (\ref{eq:eqgt2pr}).
\end{itemize}
\end{theorem}
\begin{proof}
For part (i), we may use the analysis of
Theorem~\ref{theorem:pg1}; we have for some $C>0$
\[
\mathbb{E}(|X(t)|)\leq C\; \mathbb{E}(|X_0|)\; (1+t)^{\alpha},
\]
where $\alpha$ satisfies $a+|b|q^{\alpha}=0$. If $a+|b|<0$, then
we must have $\alpha<0$, proving the result.

For part (ii), we revisit the method of proof of Theorem
\ref{theorem:pg1}, part (ii). If
$2a+\sigma^2+\rho^2+2|b+\sigma\rho|<0$, there exists $\eta>0$ such
that
\[
2a+\sigma^2+\eta^2|b+\sigma\rho|+
\frac{1}{\eta^2}|b+\sigma\rho|+\rho^2<0.
\]
In this case, the optimal choice of $\eta$ which minimises
$\alpha=\alpha(\eta)$ given in (\ref{eq:eqgt2}), yields a constant
$\alpha<0$ given by (\ref{eq:eqgt2pr}).
\end{proof}

\begin{remark}    \label{remark:rpg2}
In the case $\rho=0$, the solution of (\ref{stochpanto}) is
asymptotically stable in first mean when the corresponding
deterministic problem is asymptotically stable. When $\rho\not=0$,
observe that $2(a+|b|)+(|\sigma|+|\rho|)^2<0$ implies
$2a+\sigma^2+\rho^2+2|b+\sigma\rho|<0$. Therefore, the solution of
(\ref{stochpanto}) is asymptotically stable in mean square
whenever the deterministic problem is asymptotically stable,
provided the noise coefficients $\sigma$ and $\rho$ are not too
large.
\end{remark}

\subsection{Almost sure polynomial growth bounds} \label{aspolgro}
We can use the estimates of the previous section to determine the
polynomial asymptotic behaviour of the solution of
(\ref{stochpanto}) in an almost sure sense. In Theorem
\ref{theorem:pg3}, we concentrate on the a.s. polynomial
boundedness, while Theorem \ref{theorem:pb1} concerns the a.s.
polynomial stability of solutions of (\ref{stochpanto}). Again we
treat in both propositions the cases $\rho=0$ and $\rho\neq 0$
separately, obtaining sharper results in the former case.

We note by the Burkholder-Davis-Gundy inequality that there exists
a positive, universal, constant $c_2$ (universal in the sense that
it is independent of the process $X$ and times $t_1$, $t_2$) such
that
\begin{equation}
\label{bdg} \mathbb{E}\left[ \sup_{t_1\leq s\leq t_2}
\left(\int_{t_1}^{s} X(u) \,dB(u)\right)^2 \right] \leq c_2
\mathbb{E}\left[\int_{t_1}^{t_2} X^2(s)\,ds \right].
\end{equation}
This bound is of great utility in obtaining estimates on the
expected value of the suprema of the process.

We now proceed with the main results of this section.
\begin{theorem}  \label{theorem:pg3}
Let $(X(t))_{t \geq 0}$ be the process uniquely defined by
(\ref{stochpanto}).
\begin{itemize}
\item[(i)]
Let $\rho=0$, $\mathbb{E}(|X_0|^2)<\infty$. If $a<0$, then
\[
\limsup_{t \rightarrow \infty} \frac{\log|X(t)|}{\log t}
\leq\alpha + 1, \quad \text{a.s.},
\]
where $\alpha$ is defined by (\ref{eq:eqgt1}).
\item[(ii)]
Let $\rho\not=0$, $\mathbb{E}(|X_0|^4)<\infty$. If
$2a+\sigma^2<0$, then
\[
\limsup_{t \rightarrow \infty} \frac{\log|X(t)|}{\log t}
\leq\frac{1}{2}(\alpha + 1), \quad \text{a.s.},
\]
where $\alpha$ is defined by (\ref{eq:eqgt2pr}).
\end{itemize}
\end{theorem}
\begin{proof}
In the proof of part (i) of Theorem \ref{theorem:pg1}, we
established the existence of $C>0$ and $\alpha$ (determined by
(\ref{eq:eqgt1})), so that for $t\geq 0$ we have 
\begin{equation*}   
\mathbb{E}(|X(t)|)\leq C\; \mathbb{E}(|X_0|)\;(1+ t)^{\alpha} ,
\end{equation*}
whenever $a<0$. Therefore for $t\geq 1$ we have with $C^\ast=C\max(1,2^\alpha)\mathbb{E}(|X_0|)$ 
\begin{equation}   \label{eq:eqpe1}
\mathbb{E}(|X(t)|)
\leq C\; \mathbb{E}(|X_0|)\;t^\alpha \cdot (1+ 1/t)^{\alpha} \leq C^\ast\;t^\alpha.
\end{equation}
Take $\varepsilon>0$ to be fixed, and define $\lambda=-(\alpha+ 1
+\varepsilon)$. Set $a_n=n\tau$, where
$\tau^{\frac{1}{2}}c_2|\sigma|=\frac{1}{2}$, and $c_2$ is the
constant in (\ref{bdg}), so that
\begin{equation}         \label{EQ4}
c_2 (a_{n+1}-a_{n})^{\frac{1}{2}} |\sigma| = \frac{1}{2}.
\end{equation}
Moreover, we have
\begin{equation}   \label{eq:eqc2}
\int_{a_n}^{a_{n+1}} s^{\alpha}\,ds \leq \tau(\tau
n)^{\alpha}\max(1,2^{\alpha}).
\end{equation}
To see this, note for $n \geq 1$ that
\[
\frac{1}{a_{n+1}-a_n} \int_{a_n}^{a_{n+1}} \left( \frac{s}{a_n}
\right)^{\alpha} \,ds \leq \max_{a_n\leq s \leq a_{n+1}} \left(
\frac{s}{a_n} \right)^{\alpha}.
\]
The right hand side of this expression equals 1 if $\alpha\leq0$,
while for $\alpha>0$, it is $(a_{n+1}/a_n)^{\alpha}$. However,
$(a_{n+1}/a_n)^{\alpha}\leq2^{\alpha}$, for $\alpha>0$ and $n\geq
1$.

For every $t \in \mathbb{R}^+$, there exists $n\in \mathbb{N}$
such that $a_n \leq t < a_{n+1}$. Consider $t \in [a_n, a_{n+1})$,
so
\[
X(t) = X(a_n) + \int_{a_n}^{t} \{ a X(s) + b X(qs)\} \,ds +
 \int_{a_n}^t \sigma  X(s)\,dB(s).
\]
Using the triangle inequality, taking suprema and expectations
over $[a_n,a_{n+1}]$, and noting that $X(t)$ is continuous, we
arrive at
\begin{multline}    \label{eq:eqpe3}
\mathbb{E}\left( \sup_{a_n \leq t \leq a_{n+1}} |X(t)| \right)
 \leq \mathbb{E}(| X(a_n)|)  + \int_{a_n}^{a_{n+1}} \{| a|\;
 \mathbb{E}(|X(s)|) + |b|\; \mathbb{E}(| X(qs)|)\} \,ds \\
+ \mathbb{E} \left( \sup_{a_n \leq t \leq a_{n+1}}
          \left| \int_{a_n}^t \sigma  X(s)\,dB(s)\right|\right).
\end{multline}
Using (\ref{bdg}) and (\ref{EQ4}), we can bound the last term on
the right hand side of (\ref{eq:eqpe3}) for $n\geq 1$ by
\begin{eqnarray}
\lefteqn{\mathbb{E} \left(\sup_{a_n \leq t \leq a_{n+1}}
          \left| \int_{a_n}^t \sigma  X(s)\,dB(s)\right|\right)} &&
\nonumber\\
& \leq & c_2 \; \mathbb{E} \left(\int_{a_n}^{a_{n+1}} \sigma^2
X^2(s)\,ds \right)^{\frac{1}{2}}
\leq  c_2\; \mathbb{E} \left((a_{n+1} - a_n)  \sigma^2
     \sup_{a_n \leq t \leq a_{n+1}}  X^2(t)\right)^{\frac{1}{2}}
 \nonumber\\
&=&\frac{1}{2}\ \mathbb{E}\left(\sup_{a_n \leq t \leq a_{n+1}}
|X(t)| \right). \label{eq:eqpe4}
\end{eqnarray}
Hence, for $n$ sufficiently large, from (\ref{eq:eqpe1}), (\ref{eq:eqc2}),
(\ref{eq:eqpe3}) and  (\ref{eq:eqpe4}) we obtain
\begin{eqnarray*}
\lefteqn{\mathbb{E}(\sup_{a_n \leq t\leq a_{n+1}} |X(t)|)} \\
&\leq& 2\ \mathbb{E}(|X(a_n)|) +  2 \int_{a_n}^{a_{n+1}}
\{|a|\;\mathbb{E}(|X(s)|) + |b|\;\mathbb{E}(|X(qs)|)\}\,ds \\
&\leq& 2\;C^\ast\;a_n^{\alpha} + 2\;(|a|C^\ast+|b|C^\ast q^{\alpha})
\int_{a_n}^{a_{n+1}} s^{\alpha}\,ds \leq\tilde{C}\;n^{\mu\alpha},
\end{eqnarray*}
where $\tilde{C}=2\;C^\ast\tau^{\alpha}+2\;C^\ast\tau^{\alpha+1}
(|a|+|b|q^{\alpha})\max(1,2^{\alpha})$. 
Thus, by Markov's inequality, for every $\gamma>0$, we have
\begin{eqnarray*}
\mathbb{P}\left( \sup_{a_n \leq t\leq a_{n+1}} |X(t)|
n^{-\alpha-1-\varepsilon}\geq\gamma\right) &\leq&
\frac{1}{\gamma}\; \mathbb{E}\left(\sup_{a_n \leq t\leq a_{n+1}}
|X(t)|\; n^{-\alpha-1-\varepsilon}\right) \\
&\leq&
\frac{1}{\gamma}\frac{1}{n^{1+\varepsilon}}\frac{1}{n^{\alpha}}
\cdot\tilde{C}n^{\alpha} =
\frac{\tilde{C}}{\gamma}\frac{1}{n^{1+\varepsilon}}.
\end{eqnarray*}
Therefore, by the first Borel-Cantelli Lemma, we must have
\begin{equation}   \label{eq:eqpe5}
\lim_{n \rightarrow \infty} \sup_{a_n\leq t \leq a_{n+1}}
|X(t)|\;n^{\lambda} =0\quad \text{a.s.}
\end{equation}
For each $t\in \mathbb{R}^{+}$, there exists $n(t)\in \mathbb{N}$
such that $\tau\; n(t) \leq t < \tau(n(t)+1)$. Therefore,
\begin{equation}  \label{eq:eqpe6}
\lim_{t \rightarrow \infty} \frac{t}{\tau\; n(t)} = 1.
\end{equation}
Using (\ref{eq:eqpe5}) and (\ref{eq:eqpe6}), we have
\begin{eqnarray*}
\limsup_{t \rightarrow \infty} |X(t)|\; t^{\lambda} &\leq&
\limsup_{t \rightarrow \infty} \sup_{\tau\; n(t)\leq s\leq
\tau(n(t)+1)} |X(s)| (\tau \; n(t))^{\lambda} \cdot \limsup_{t
\rightarrow \infty}
\left(\frac{t}{\tau\; n(t)}\right)^{\lambda} \\
&=&\limsup_{n \rightarrow \infty} \sup_{a_n \leq s \leq a_{n+1}}
|X(s)|\; n^{\lambda} \cdot \limsup_{t \rightarrow \infty}
\left(\frac{t}{\tau\; n(t)}\right)^{\lambda} =0, \quad \text{a.s.}
\end{eqnarray*}
Therefore
\[
\limsup_{t \rightarrow \infty} \frac{\log|X(t)|}{\log t}
\leq-\lambda = \alpha+1+\varepsilon, \quad \text{a.s.}
\]
Letting $\varepsilon\downarrow 0$ through the rationals completes
the proof.

The proof of part (ii) is very similar, and the outline only will
be sketched here. By hypothesis, Theorem \ref{theorem:pg1} tells
us that there exists $\alpha$ satisfying (\ref{eq:eqgt2}) such
that
\[
\mathbb{E}(|X(t)|^{2})\leq C\; \mathbb{E} (|X_0|^2) \ (1+ t)^{\alpha}.
\]
Once again, this means for all $t\geq 1$ we have 
\[
\mathbb{E}(|X(t)|^{2})\leq C^\ast\;t^{\alpha}.
\]
Define, for every $\varepsilon>0$,
$\lambda=-(\alpha+1+\varepsilon)$ and $\varepsilon\in (0,1)$ as
before.  Define $a_n$ as above. To get a bound on
$\mathbb{E}(\sup_{a_n\leq t \leq a_{n+1}} X^{2}(t))$, proceed as
for equation (\ref{eq:eqpe3}) to obtain
\begin{multline}  \label{eq:eqpe7}
\mathbb{E}\left( \sup_{a_n \leq t \leq a_{n+1}} X^2(t) \right)
 \leq 3\;\bigg\{\mathbb{E}(X^2(a_n)) +
\mathbb{E}\left( \int_{a_n}^{a_{n+1}}
|aX(s)+ b X(qs)| \,ds \right)^2\\
+ \mathbb{E} \left[ \sup_{a_n \leq t \leq a_{n+1}}
 \left( \int_{a_n}^t (\sigma X(s)+\rho X(qs))\,dB(s)\right)^{2}
\right] \bigg\}.
\end{multline}
By (\ref{eq:eqc2}), the second term on the right hand side of
(\ref{eq:eqpe7}) has as a bound (for $n\geq 1$):
\begin{align*}
\mathbb{E}\left( \int_{a_n}^{a_{n+1}} |aX(s)+ b X(qs)| \,ds
\right)^2
&\leq (2a^{2}C^\ast+2b^{2}C^\ast q^{\alpha})(a_{n+1}-a_{n})
\int_{a_n}^{a_{n+1}} s^{\alpha}\,ds \\
&\leq (2a^{2}C^\ast +2b^{2}C^\ast q^{\alpha})\max(1,2^{\alpha})\cdot
\tau^{\alpha+1}n^{\alpha}.
\end{align*}
Using the Burkholder-Davis-Gundy inequality and (\ref{eq:eqc2}),
the third term on the right hand side of (\ref{eq:eqpe7}) has as a
bound (for $n\geq1$):
\begin{eqnarray*}
\lefteqn{ \mathbb{E} \left[ \sup_{a_n \leq t \leq a_{n+1}}
 \left( \int_{a_n}^t (\sigma X(s)+\rho X(qs))\,dB(s)\right)^{2}
\right]
} \\
&\leq& c_2\;\mathbb{E}\left[
\int_{a_n}^{a_{n+1}} (\sigma X(s)+\rho X(qs))^{2}\,ds\right] \\
&\leq& c_2\;(2\sigma^{2}C^\ast+2\rho^{2}C^\ast q^{\alpha})
(a_{n+1}-a_{n}) \int_{a_n}^{a_{n+1}} s^{\alpha}\,ds \\
&\leq& c_2\;(2\sigma^{2}C^\ast+2\rho^{2}C^\ast q^{\alpha})\max(1,2^{\alpha})
\cdot \tau^{\alpha+1} n^{\alpha}.
\end{eqnarray*}
Inserting these estimates into (\ref{eq:eqpe7}), we obtain
\[
\mathbb{E}\left( \sup_{a_n \leq t \leq a_{n+1}} X^2(t) \right)
\leq\tilde{C}\;n^{\alpha},
\]
where $\tilde{C}$ is some $n$-independent constant. The rest of
the proof goes through as above.
\end{proof}

The results on polynomial stability in the first and second mean
can be used to establish almost sure stability of solutions of
(\ref{stochpanto}) with a polynomial upper bound on the decay
rate.

\begin{theorem}    \label{theorem:pb1}
Let $(X(t))_{t \geq 0}$ be the process uniquely defined by
(\ref{stochpanto}).
\begin{itemize}
\item[(i)]
Let $\rho=0$,  $\mathbb{E}(|X_0|^2)<\infty$. If $a+ |b|/q <0$,
then $\alpha$, defined by (\ref{eq:eqgt1}), satisfies $\alpha<-1$,
and we have
\[
\limsup_{t \rightarrow \infty} \frac{\log|X(t)|}{\log t} \leq
\alpha + 1, \quad \text{a.s.},
\]
so $X(t)\rightarrow0$ as $t\rightarrow \infty$, a.s.
\item[(ii)]
Let $\rho\not=0$,  $\mathbb{E}(|X_0|^4)<\infty$. If
\begin{equation}   \label{eq:eqpb1}
2a+\sigma^2+\frac{\rho^2}{q}+\frac{2}{\sqrt{q}}|b+\sigma\rho|<0,
\end{equation}
then $\alpha$, defined by (\ref{eq:eqgt2pr}), satisfies
$\alpha<-1$, and we have
\[
\limsup_{t \rightarrow \infty} \frac{\log|X(t)|}{\log t} \leq
\frac{1}{2}(\alpha + 1), \quad \text{a.s.},
\]
so $X(t)\rightarrow0$ as $t\rightarrow \infty$, a.s.
\end{itemize}
\end{theorem}
\begin{proof}
In part (i) of the theorem, we need $\alpha$ defined by
(\ref{eq:eqgt1}) to satisfy $\alpha<-1$. Considering
(\ref{eq:eqgt1}), we see that $a+|b|/q<0$ implies $\alpha<-1$.

The proof of part (ii) follows along identical lines. Suppose
(\ref{eq:eqpb1}) holds. If we choose $\eta^2=1/\sqrt{q}$, then
\[
2a+\sigma^2+\eta^2|b+\sigma\rho| +
\left(\rho^2+\frac{1}{\eta^2}|b+\sigma\rho|\right)\frac{1}{q}<0,
\]
and so $\alpha$ defined by (\ref{eq:eqgt2pr}) satisfies
$\alpha<-1$.
\end{proof}

\section{Exponential Upper Bounds} \label{expupper}
In~\cite{Mao1}, it is shown that stochastic delay differential
equations, or stochastic functional differential equations with
bounded delay, are a.s. bounded by increasing exponential
functions, provided that the coefficients of the equation satisfy
global linear bounds. More precisely, Mao shows that the top
Lyapunov exponent is bounded almost surely by a finite constant.
For the stochastic pantograph equation, we similarly show that all
solutions have top a.s. and $p$-th mean ($p=1,2$) Lyapunov
exponents which are bounded by finite constants. We consider only
parameter regions in which the polynomial boundedness of the
solution of (\ref{stochpanto}) has not been established, as the
a.s. exponential upper bound (respectively, the $p$-th mean
exponential upper bound) is a direct consequence of the a.s.
polynomial boundedness (respectively, $p$-th mean polynomial
boundedness) of the solution.

Therefore, when $a>0$ and $\rho=0$, we show that there is an
exponential upper bound on solutions in a first mean and almost
sure sense; when $a<0$ and $\rho=0$, we already know that there is
a polynomial bound in first mean and almost surely. When
$\rho\not=0$, and $2a+\sigma^2>0$, we show that there is an
exponential upper bound on solutions in a mean square and almost
sure sense; when $2a+\sigma^2<0$ and $\rho\not=0$, we have already
established that there is a polynomial mean square and almost sure
bound on solutions. However, because these are upper bounds, we
cannot say whether the only classes of asymptotic behaviour are of
polynomial or exponentially growing type. Nonetheless, these are
the only classes of behaviour exhibited by deterministic
pantograph equations, and, in a later work (where $\rho=0$), we
establish that solutions are either exponentially growing, or
cannot grow (or decay) as fast as any exponential.

In Theorem \ref{theorem:pexp1} below, we obtain an exponential
upper bound on the $p$-th mean ($p=1,2$) using the comparison
principle obtained earlier in this paper, and obtain almost sure
exponential upper bounds by using ideas of Theorem 4.3.1
in~\cite{Maobk2} for stochastic differential equations, which are
developed in Theorem 5.6.2 in~\cite{Mao1} for stochastic equations
with delay.
\begin{theorem}   \label{theorem:pexp1}
Let $(X(t))_{t\geq 0}$ be the process satisfying
(\ref{stochpanto}).
\begin{itemize}
\item[(i)]
Let $\rho=0$, $\mathbb{E}(|X_0|^2)<\infty$. If $a>0$, then
\[
\mathbb{E}(|X(t)|)\leq C\ \mathbb{E}(|X_0|) \ e^{at},
\]
and
\[
\limsup_{t \rightarrow \infty} \frac{1}{t}\log |X(t)| \leq a \quad
\text{a.s.}
\]
\item[(ii)]
Let $\rho\not=0$, $\mathbb{E}[|X_0|^2]<\infty$. If
$2a+\sigma^{2}>0$, then
\[
\mathbb{E}(|X(t)|^2)\leq C\ \mathbb{E}(|X_0|^2) \
e^{(2a+\sigma^2)\;t},
\]
and
\[
\limsup_{t \rightarrow \infty} \frac{1}{t}\log |X(t)| \leq
a+\frac{1}{2}\sigma^{2} \quad \text{a.s.}
\]
\end{itemize}
\end{theorem}
\begin{proof}
The bounds on expectations follow from Lemma \ref{lemma:lemp3}
part (ii). For part (i) (where $\rho=0$) the proof of equation
(\ref{eq:eqcomp3}) in Theorem~\ref{theorem:pg1} part (i), together
with part (ii) of Lemma~\ref{lemma:lemp3} gives
\[
\mathbb{E}(|X(t)|)=m(t)\leq C\;m(0)\;e^{at} =
C\;\mathbb{E}(|X_0|)\;e^{at}
\]
for some $C>0$.

For part (ii) (where $\rho\not=0$) the proof of (\ref{eq:eqcomp4})
in Theorem~\ref{theorem:pg1} part (ii) gives
\[
\mathbb{E}(|X(t)|^2)=\widetilde{m}(t) \leq
C\;\widetilde{m}(0)\;e^{\bar{a}t} =
C\;\mathbb{E}(|X(0)|^2)\;e^{\bar{a}t}, \quad t\geq 0
\]
for some $C>0$, where $\bar{a}=2a +\sigma^2 +
|b+\sigma\rho|\eta^{2}$. As $\eta$ can be chosen arbitrarily
small, we have
\begin{equation}   \label{eq:eqexp2}
\mathbb{E}(|X(t)|^2)\leq C(\varepsilon)\;e^{(2a
+\sigma^2+\varepsilon)t} \; \mathbb{E}(|X(0)|^2)
\end{equation}
for every $\varepsilon>0$. This gives the desired result.

This analysis now enables us to obtain an exponential upper bound
on
\[
t \mapsto \mathbb{E} ( \sup_{0\leq s \leq t} X^2(s)).
\]
The proof follows the idea of the proof of Theorem 4.3.1 in Mao
\cite{Maobk2}, and related results in Mao \cite{Mao1}. It also is
similar to earlier results in this paper, so we present an outline
for part (ii) only.

Using the inequality $(a+b+c)^2 \leq 3(a^2 + b^2 + c^2)$, we get
\begin{multline*}
X^2(t) \leq 3 \; X^2(0) + 3\  \left(\int_0^t \{a X(s) + b X(qs)\} \,ds
\right)^2
\\ + 3 \ \left(\int_0^t \{\sigma X(s) + \rho X(qs)\}\,dB(s)\right)^2 ,
\end{multline*}
so using the Cauchy-Schwarz inequality, taking suprema, and then
expectations on both sides of the inequality, we arrive at
\begin{multline*}
\mathbb{E}(\sup_{0\leq t \leq T} X^2(t))
 \leq 3 \  \mathbb{E}( X^2(0)) + 3\ T\int_0^T
\mathbb{E}(|a X(s)
+ b X(qs)|)^{2} \,ds \\
 + 3 \ \mathbb{E} \left[\sup_{0\leq t \leq T}
\left(\int_0^t \{\sigma X(s)+ \rho X(qs)\}\,dB(s)\right)^2\right] ,
\end{multline*}
for any $T>0$. The second term on the right-hand side can be
bounded using the inequality $(a+b)^2 \leq 2(a^2 +b^2)$. The third
term can be bounded by the Burkholder-Davis-Gundy inequality.
Therefore, by (\ref{bdg}), there exists a $T$-independent constant
$c_2>0$ such that
\begin{multline*}
\mathbb{E}(\sup_{0\leq t \leq T} X^2(t)) \leq 3 \ \mathbb{E}
(X^2(0)) + 3 \ T\ \int_0^T \{2a^2 \mathbb{E} |X(s)|^2
 + 2b^2 \mathbb{E} |X(qs)|^2\} \,ds  \\
 + 3 \ c_2 \int_0^T \{2\sigma^2 \mathbb{E} |X(s)|^2
 + 2 \rho^2 \mathbb{E} |X(qs)|^2\}\,ds .
\end{multline*}
Noting that $q\in(0,1)$, we now appeal to (\ref{eq:eqexp2}) to
show for each fixed $\varepsilon >0$ that there exists a
$C_3(\varepsilon)>0$ such that
\[
\mathbb{E}(\sup_{0\leq t \leq T} X^2(t)) \leq
C_3(\varepsilon)\;(T+1)\; \exp\left(( 2 |a|+\sigma^2+\varepsilon)
T \right).
\]
Let $\lambda= |a|+ \frac{1}{2}\sigma^2$, so
that for each fixed $\varepsilon >0$ we have
\[
\mathbb{E} \left[ \left(\frac{\sup_{0\leq s \leq t} |X(s)|}
{e^{(\lambda+\varepsilon)t}}\right)^2 \right] \leq
C_3(\varepsilon)\;(t+1)\; \exp( -\varepsilon t).
\]
Using Chebyshev's inequality and the first Borel-Cantelli Lemma,
we see that for every fixed $\varepsilon\in (0,1)$ there is an almost 
sure event $\Omega_\varepsilon$ such that
\[
\limsup_{t \rightarrow \infty} |X(t,\omega)|\;{e^{-(\lambda+\varepsilon)t}} 
\leq e^{\lambda+1}, \quad \text{for $\omega\in \Omega_\varepsilon$.}
\]
Therefore for $\omega\in \Omega_{\varepsilon}$
\[
\limsup_{t \rightarrow \infty} \frac{1}{t} \log|X(t,\omega)| \leq \lambda
+ \varepsilon.
\]
Let $\Omega^{\ast} = \cap_{n\in\mathbb{N}} \Omega_{1/n}$. Then 
$\mathbb{P}[\Omega^\ast]=1$ and for all $\omega\in\Omega^\ast$ and all
$n\in \mathbb{N}$, we have
\[
\limsup_{t\rightarrow\infty} \frac{1}{t}\log |X(t,\omega)| 
\leq \lambda+\frac{1}{n}.
\]
Hence 
\[
\limsup_{t\rightarrow\infty} \frac{1}{t}\log |X(t,\omega)| \leq \lambda
\]
for all $\omega\in \Omega^\ast$, as required.
\end{proof}

\section{Generalisations of the scalar stochastic pantograph equation}
In this section, we consider some generalisations of the scalar
stochastic pantograph equation, and show, with some restrictions
and additional analysis, that our stability and asymptotic
analysis extends to cover these more general problems.

In the first subsection, we consider the scalar stochastic
pantograph equation with arbitrarily many proportional delays
\begin{equation}  \label{eq:eq71}
dX(t)=\left(aX(t)+ \sum_{i=1}^{n} b_{i}X(q_{i}t)\right)\,dt +
\left(\sigma X(t) + \sum_{i=1}^{m}
\sigma_{i}X(r_{i}t)\right)\,dB(t).
\end{equation}
where $q_{i}\in (0,1)$ for $i=1,\ldots,n$ and $r_i \in(0,1)$ for
$i=1,\ldots,m$ are sets of distinct real numbers, $a,\sigma\in
\mathbb{R}$, $b_{i}\in \mathbb{R}$ for $i=1,\ldots,n$, and
$\sigma_i\in \mathbb{R}$ for $i=1,\ldots,m$. We assume $(B(t))_{t
\geq 0}$ is a standard one-dimensional Brownian motion. The
equation (\ref{eq:eq71}) is a stochastic version of the pantograph
equation
\begin{equation}    \label{eq:eq72}
x'(t)=ax(t)+\sum_{i=1}^{\infty} b_i x(q_{i}t), \quad t \geq 0
\end{equation}
studied by Liu~\cite{Liu}. The asymptotic behaviour of
(\ref{eq:eq72}) as $t\rightarrow \infty$ bears many similarities
to that of (\ref{detpanto}), and we will use results about the
asymptotic behaviour of (\ref{eq:eq72}) in order to obtain
estimates on the asymptotic behaviour of (\ref{eq:eq71}).

In the second subsection, we study the finite dimensional
stochastic pantograph equation
\begin{equation}    \label{eq:eq73}
dX(t) = \left(AX(t)+BX(qt)\right)\,dt + \left(\Sigma X(t)+ \Theta
X(qt)\right)\,dB(t),
\end{equation}
where $A$, $B$, $\Sigma$, $\Theta \in M_{d,d}(\mathbb{R})$, the
set of all square $d\times d$ matrices with real entries. To
guarantee a unique solution, we must specify $\mathbb{E} (\|X_0\|^2)<
\infty$. Here $\|x\|$ denotes the Euclidean norm for
$x\in\mathbb{R}^{d}$. We denote by $\langle x, y\rangle$ the
standard inner product on for $x,y \in\mathbb{R}^d$, and, with a
slight abuse of notation, $\|C\|$ for the matrix norm of $C  \in
M_{d,d}(\mathbb{R})$, where the norm is induced from the standard
Euclidean norm. Equation (\ref{eq:eq73}) is a stochastic
complement to the finite dimensional deterministic pantograph
equation studied by Carr and Dyson~\cite{CarrD2}.

\subsection{Scalar equation with many delays}
We need some preliminary results in order to obtain asymptotic
estimates on (\ref{eq:eq71}).

Consider the equation
\begin{equation}    \label{eq:eq74}
x'(t)=\tilde{a}x(t)+\sum_{i=1}^{n} |\tilde{b}_{i}|x(p_{i}t),
\end{equation}
where $\tilde{a}<0$ and $\tilde{b}_{i}\in \mathbb{R}$. Define the
set
\begin{equation}    \label{eq:eq75}
\tilde{E}=\{z\in\mathbb{C}\::\: \tilde{a}+
\sum_{i=1}^{n}|\tilde{b}_{i}|p_{i}^{z}=0\},
\end{equation}
and
\begin{equation}     \label{eq:eq76}
\tilde{\alpha} = \sup\{\Re e(z)\::\: z\in \tilde{E}\}.
\end{equation}
Then
\begin{lemma}    \label{lemma:l71}
Suppose that at least one $|\tilde{b}_i|\neq 0$ in
(\ref{eq:eq74}). Then, there exists $\alpha\in \tilde{E}\cap
\mathbb{R}$ such that $\alpha=\tilde{\alpha}$. Moreover,
$\alpha<0$ if and only if
\begin{equation}    \label{eq:eq77}
\tilde{a}+ \sum_{i=1}^{n}|\tilde{b}_{i}|<0,
\end{equation}
and $\alpha<-1$ if and only if
\begin{equation}     \label{eq:eq78}
\tilde{a}+ \sum_{i=1}^{n} \frac{1}{p_i}|\tilde{b}_{i}| < 0.
\end{equation}
\end{lemma}
\begin{proof}
Let $f(z)=a+\sum_{i=1}^{n} |\tilde{b}_i| p_i^{z}$. On
$\mathbb{R}$, $f$ is decreasing and continuous, so as
$\lim_{z\rightarrow -\infty} f(z) = \infty$, and
$\lim_{z\rightarrow \infty} f(z)=a<0$, $f$ has a unique real zero
at $\alpha$. Clearly, $\alpha<0$ if and only if (\ref{eq:eq77})
holds, and $\alpha<-1$ if and only if (\ref{eq:eq78}) is true.

To show that $\alpha$ is greater than or equal to the real part of
any other member of $\tilde{E}$, consider $z\in \tilde{E}$ with
$z=z_1+i z_2$. Then
\[
\Re e (f(z)) = a + \sum_{i=1}^{n}
|\tilde{b}_i|p_i^{z_1}\cos(z_2\log p_i) \leq f(z_1).
\]
If $z_1>\alpha$, then $0=\Re e(f(z)) = f(z_1)<f(\alpha)=0$, a
contradiction. Hence $\tilde{\alpha}=\alpha$.
\end{proof}

The relevance of $\tilde{\alpha}$ defined by (\ref{eq:eq76}) is
explained by Theorem 2.14 in~\cite{Liu}. There, it is shown that
every solution of (\ref{eq:eq74}) satisfies
\begin{equation}    \label{eq:eqliuest}
x(t) = \mathcal{O}(t^{\beta}), \quad \mbox{as $t\rightarrow
\infty$}
\end{equation}
for any $\beta>\tilde{\alpha}$, where $\tilde{\alpha}$ is defined
by (\ref{eq:eq76}). As a consequence of Lemma \ref{lemma:l71} and
(\ref{eq:eqliuest}), we obtain the following result.
\begin{lemma}    \label{lemma:l72}
Let $\alpha$ be the real member of $\tilde{E}$ defined by
(\ref{eq:eq75}), and $x$ be any solution of (\ref{eq:eq74}). 
Then we have
\begin{equation}    \label{eq:eq79}
\limsup_{t \rightarrow \infty} \frac{\log|x(t)|}{\log t}\leq
\alpha.
\end{equation}
\end{lemma}
We use these results to obtain asymptotic estimates on a
parameterised family of equations of the form (\ref{eq:eq74}).

Towards this end, let $(\nu_{i})_{i=1,\ldots,n}$,
$(\mu_{i})_{i=1,\ldots,m}$, $(\lambda_{i})_{i=1,\ldots,m}$ be as
yet unspecified sequences of positive real numbers, and define the
equation
\begin{multline}      \label{eq:eq7A}
y'(t) = \left(2a+\sigma^2+\sum_{i=1}^{n} |b_i|{\nu_i}^{2} +
\sum_{j=1}^{m} |\sigma||\sigma_j|{\lambda_{j}}^2\right)y(t)
+\sum_{i=1}^{n} |b_{i}|\frac{1}{{\nu_i}^2}y(q_{i}t) \\
+\sum_{l=1}^{m}
{\sigma_{l}}^2{\mu_{l}}^{-2}\cdot\sum_{j=1}^{m}{\mu_{j}}^{2}y(r_jt)
+\sum_{j=1}^{m} |\sigma||\sigma_j|\frac{1}{{\lambda_j}^2}y(r_jt).
\end{multline}
We will write $y$ above as $y_{\nu,\lambda,\mu}$.
\begin{lemma}     \label{lemma:l7A}
Suppose that $y=y_{\nu,\lambda,\mu}$ is a solution of
(\ref{eq:eq7A}). Then each of the following is true:
\begin{itemize}
\item[(i)] If
\begin{equation}     \label{eq:eq7B}
2a+\sigma^2<0,
\end{equation}
there exists
$(\nu,\lambda,\mu)=(\nu^{\ast},\lambda^{\ast},\mu^{\ast})$ and
$\alpha\in \mathbb{R}$ such that
\begin{equation}     \label{eq:eq7C}
\limsup_{t \rightarrow\infty} \frac{\log
|y_{\nu^{\ast},\lambda^{\ast},\mu^{\ast}}(t)|}{\log t} \leq
\alpha.
\end{equation}
\item[(ii)] If
\begin{equation}     \label{eq:eq7D}
2a+2\sum_{i=1}^{n} |b_i| +(|\sigma|+\sum_{j=1}^{m}|\sigma_j|)^2<0,
\end{equation}
there exists
$(\nu,\lambda,\mu)=(\nu^{\ast},\lambda^{\ast},\mu^{\ast})$ and
$\alpha<0$ such that (\ref{eq:eq7C}) holds.
\item[(iii)] If
\begin{equation}      \label{eq:eq7E}
2a+ 2 \sum_{i=1}^{n} \frac{|b_i|}{\sqrt{q_i}}
+\left(|\sigma|+\sum_{j=1}^{m}
\frac{|\sigma_j|}{\sqrt{r_j}}\right)^2<0,
\end{equation}
there exists
$(\nu,\lambda,\mu)=(\nu^{\ast},\lambda^{\ast},\mu^{\ast})$ and
$\alpha<-1$ such that (\ref{eq:eq7C}) holds.
\end{itemize}
\end{lemma}
\begin{proof}
Notice that (\ref{eq:eq7A}) conforms to the form of
(\ref{eq:eq74}). To establish (i), we see by comparing
(\ref{eq:eq7A}) with (\ref{eq:eq74}) that the conclusion of Lemma
\ref{lemma:l72} can be applied to $y=y_{\nu,\lambda,\mu}$ provided
\begin{equation}    \label{eq:eq7F}
2a+\sigma^2+\sum_{i=1}^{n} |b_{i}|{\nu_{i}}^2 +
\sum_{j=1}^{m}|\sigma||\sigma_j|{\lambda_{j}}^2<0.
\end{equation}
By choosing the $\nu$'s and $\lambda$'s arbitrarily small, we see
that (\ref{eq:eq7B}) suffices to establish (\ref{eq:eq7F}), and
thereby (\ref{eq:eq7C}) for $\alpha\in\mathbb{R}$. To prove (ii),
by analogy to (\ref{eq:eq75}) and (\ref{eq:eq76}), we introduce
for $y=y_{\nu,\lambda,\mu}$ the set
\begin{multline}    \label{eq:eq710}
E(\nu,\lambda,\mu) =\Big\{z\in\mathbb{C}\::\: 2a+
\sigma^2+\sum_{i=1}^{n}|b_i|{\nu_{i}}^2
+\sum_{j=1}^{m}|\sigma||\sigma_j|{\lambda_{j}}^2
+\sum_{i=1}^{n}|b_i|\frac{1}{{\nu_i}^2}{q_i}^z \\
+\sum_{l=1}^{m} {\sigma_{l}}^2{\mu_{l}}^{-2}
\cdot\sum_{j=1}^{m}{\mu_{j}}^2{r_{j}}^{z} +\sum_{j=1}^{m}
|\sigma||\sigma_j|\frac{1}{{\lambda_j}^2}{r_{j}}^z=0 \Big\},
\end{multline}
and let
\begin{equation}     \label{eq:eq711}
\alpha(\nu,\lambda,\mu)=\sup\{\Re e(z)\::\: z\in
E(\nu,\lambda,\mu)\}.
\end{equation}
By Lemma \ref{lemma:l71} and (\ref{eq:eq77}), we see that
$\alpha(\nu,\lambda,\nu)<0$ if
\begin{multline}    \label{eq:eq712}
2a+ \sigma^2+\sum_{i=1}^{n}|b_i|
\left({\nu_{i}}^2+\frac{1}{{\nu_i}^2}\right)
+\sum_{j=1}^{m}|\sigma||\sigma_j|
\left({\lambda_{j}}^2+\frac{1}{{\lambda_j}^2}\right)
\\
+\sum_{l=1}^{m} {\sigma_{l}}^2{\mu_{l}}^{-2}
\cdot\sum_{j=1}^{m}{\mu_{j}}^2<0.
\end{multline}
In particular, if (\ref{eq:eq7D}) is true, we see that by choosing
$\nu_{i}^{\ast}=1$, $\lambda_{j}^{\ast}=1$,
$\mu_{j}^{\ast}=\sqrt{|\sigma_{j}|}$ that (\ref{eq:eq712}) is
satisfied, and so by Lemma \ref{lemma:l72}, the estimate
(\ref{eq:eq7C}) holds with
$\alpha=\alpha(\nu^{\ast},\lambda^{\ast},\mu^{\ast})<0$.

Part (iii) follows similarly: Lemma \ref{lemma:l71} and
(\ref{eq:eq78}) ensure that $\alpha(\nu,\lambda,\mu)$ defined by
(\ref{eq:eq711}) satisfies $\alpha(\nu,\lambda,\mu)<-1$ if
\begin{multline}   \label{eq:eq713}
2a+ \sigma^2+\sum_{i=1}^{n}|b_i|
\left({\nu_{i}}^2+\frac{1}{q_i{\nu_i}^2}\right)
+\sum_{j=1}^{m}|\sigma||\sigma_j|
\left({\lambda_{j}}^2+\frac{1}{r_j{\lambda_j}^2}\right)
\\
+\sum_{l=1}^{m} {\sigma_{l}}^2{\mu_{l}}^{-2}
\cdot\sum_{j=1}^{m}{\mu_{j}}^2\frac{1}{r_j}<0.
\end{multline}
In particular, if (\ref{eq:eq7E}) is true, we see that by choosing
\[
(\nu_{i}^{\ast})^2=\frac{1}{\sqrt{q_i}}, \quad
(\lambda_{j}^{\ast})^2=\frac{1}{\sqrt{r_j}}, \quad
(\mu_{j}^{\ast})^2=\frac{|\sigma_j|}{\sqrt{r_j}}
\]
the
inequality (\ref{eq:eq713}) is satisfied, and so by Lemma
\ref{lemma:l72} the estimate (\ref{eq:eq7C}) stands with 
$\alpha=\alpha(\nu^{\ast},\lambda^{\ast},\mu^{\ast})<-1$.
\end{proof}
The following elementary inequality is also important in our
analysis.
\begin{lemma}      \label{lemma:l7B}
Suppose that $\sigma$, $x\in\mathbb{R}$, and
$(x_i)_{i=1,\ldots,m}$, $(\sigma_{i})_{i=1,\ldots,m}$ are any two
real sequences. If $(\lambda_i)_{i=1,\ldots,m}$,
$(\mu_i)_{i=1,\ldots,m}$ are two sequences of positive real
numbers, then
\begin{multline*}
\left(\sigma x + \sum_{i=1}^{m} \sigma_ix_i\right)^2 \leq
\sigma^2x^2 + \sum_{i=1}^{m} |\sigma||\sigma_i|
\left(\lambda_i^2x^2+\frac{1}{\lambda_i^2}x_i^2\right) \\
+\sum_{l=1}^{m} \sigma_l^2\mu_l^{-2} \cdot \sum_{j=1}^{m}
\mu_j^2x_j^2.
\end{multline*}
\end{lemma}
\begin{proof}
Using the inequality $2xx_i\leq
\lambda_i^2x^2+\lambda_i^{-2}x_i^2$, we get
\[
\left(\sigma x+\sum_{i=1}^{m} \sigma_ix_i\right)^2 \leq
\sigma^2x^2 + \sum_{i=1}^{m} |\sigma|\sigma_i|
\left(\lambda_i^2x^2+\frac{1}{\lambda_i^2}x_i^2\right) +
\left(\sum_{i=1}^{m} \sigma_ix_i\right)^2.
\]
Applying the Cauchy-Schwarz inequality to the last term gives
\[
\left(\sum_{i=1}^{m} \sigma_ix_i\right)^2 \leq \sum_{i=1}^{m}
\sigma_i^2\mu_i^{-2} \cdot \sum_{i=1}^{m}\mu_i^2x_i^2,
\]
proving the result.
\end{proof}
With these estimates, we can establish sufficient conditions to
ensure the polynomial asymptotic behaviour of solutions of
(\ref{eq:eq71}).

\begin{theorem}  \label{theorem:t71}
Let $(X(t))_{t \geq 0}$ be the solution of (\ref{eq:eq71}) with
$\mathbb{E}(|X_0|^4)< \infty$. Then the following holds:
\begin{itemize}
\item[(i)] Suppose $2a+\sigma^2<0$. Then there exists $C>0$ and
$\alpha\in\mathbb{R}$ such that 
\begin{equation}    \label{eq:eqt71}
\limsup_{t \rightarrow \infty}
\frac{\log\mathbb{E}[|X(t)|^2]}{\log t}\leq \alpha,
\end{equation}
and
\begin{equation}     \label{eq:eqt72}
\limsup_{t\rightarrow \infty} \frac{\log|X(t)|}{\log t}
\leq\frac{1}{2}(\alpha+1), \quad \mbox{a.s.}
\end{equation}
Therefore, the process is polynomially bounded in mean-square and
almost surely polynomially bounded.
\item[(ii)] Suppose (\ref{eq:eq7D}) holds.
Then there exists $\alpha<0$ such that (\ref{eq:eqt71})
holds. Therefore the equilibrium solution of (\ref{eq:eq71}) is
polynomially stable in mean-square.
\item[(iii)] Suppose (\ref{eq:eq7E}) holds.
Then there exists $\alpha<-1$ such that (\ref{eq:eqt71}) and
(\ref{eq:eqt72}) hold. Therefore the equilibrium solution of
(\ref{eq:eq71}) is polynomially stable in mean-square, and a.s.
polynomially stable.
\end{itemize}
\end{theorem}
\begin{proof}
Define $Y(t)=X^{2}(t)$. It\^{o}'s rule gives
\begin{multline}    \label{eq:eq714}
X^{2}(t)=X^{2}(0) +
\int_0^t 2X(s)\left(aX(s)+\sum_{i=1}^{n}b_iX(q_is)\right)\,ds \\
+ \int_{0}^{t}
\left(\sigma X(s)+ \sum_{i=1}^{m} \sigma_iX(r_is)\right)^2\,ds \\
+\int_{0}^{t} 2X(s) \left(\sigma X(s) + \sum_{i=1}^{m}
\sigma_iX(r_is)\right)\,dB(s).
\end{multline}
Let $(\nu_i)_{i=1,\ldots,n}$, $(\lambda_i)_{i=1,\ldots,m}$ be
sequences of positive real numbers. By using the inequalities
\[
2b_ixy\leq|b_i|\left(\nu_i^2x^2+\frac{1}{\nu_i^2}y^2\right), \quad
2\sigma\sigma_ixy\leq|\sigma||\sigma_i|
\left(\lambda_i^2x^2+\frac{1}{\lambda_i^2}y^2\right),
\]
for $i=1,\ldots,n$, $i=1,\ldots,m$ respectively, in conjunction
with the inequality proved in Lemma \ref{lemma:l7B}, we see that
(\ref{eq:eq714}) implies, for any $t,t+h\geq 0$ that
\begin{multline*}
Y(t+h)-Y(t)\leq\int_{t}^{t+h} \left\{2\;a\;Y(s) +\sum_{i=1}^{n} |b_i|
\left(\nu_i^2\;Y(s)+ \frac{1}{\nu_i^2}\;Y(q_is)\right)\right\}\,ds \\
+ \int_{t}^{t+h} \Big\{\sigma^2\;Y(s)
+\sum_{i=1}^{m}|\sigma||\sigma_i|
\left(\lambda_i^2\;Y(s)+\frac{1}{\lambda_i^2}\;Y(r_is)\right) \\
+\sum_{l=1}^{m} \sigma_l^2\;\mu_l^{-2}\cdot\sum_{j=1}^{m}
\mu_j^2\;Y(r_js)
\Big\}\,ds \\
+\int_{t}^{t+h} 2\;X(s) \left(\sigma X(s)+\sum_{i=1}^{m}
\sigma_iX(r_is)\right)\,dB(s).
\end{multline*}
As in earlier proofs, by defining $m(t)=\mathbb{E}(|Y(t)|)$, and
considering the function $y=y_{\nu,\lambda,\mu}$ defined by
(\ref{eq:eq7A}) with $y(0)>m(0)$, we see that $m(t)\leq
y(t)=y_{\nu,\lambda,\mu}(t)$. By making the appropriate choice of
$(\nu,\lambda,\mu)=(\nu^{\ast},\lambda^{\ast},\mu^{\ast})$ in
Lemma \ref{lemma:l7A}, we see that (\ref{eq:eq71}) is true for
some $\alpha\in\mathbb{R}$ if $2a+\sigma^2<0$; for some $\alpha<0$
if (\ref{eq:eq7D}) holds; and for some $\alpha<-1$ if
(\ref{eq:eq7E}) holds.

The proof that (\ref{eq:eqt72}) follows from (\ref{eq:eqt71})
differs little from that of the proof of Theorem
\ref{theorem:pg3}, and for this reason is omitted.
\end{proof}
We see that the polynomial boundedness of solutions of
(\ref{eq:eq71}) are ensured in mean-square and almost surely if
$2a+\sigma^2<0$, so that the terms which involve the delayed
arguments do not seem to influence the existence of polynomial
asymptotic behaviour for solutions of (\ref{eq:eq71}). We notice
that this mimics the result for solutions of the deterministic
pantograph equation with many delays, (\ref{eq:eq72}).

\subsection{Finite dimensional stochastic pantograph equation}
In~\cite{CarrD2}, Carr and Dyson considered the asymptotic
behaviour of the finite dimensional functional differential
equation
\begin{equation}    \label{eq:eq716}
x'(t)=Ax(t) + Bx(qt),
\end{equation}
as $t\rightarrow\infty$. In (\ref{eq:eq716}), $A$, $B$ are square
matrices, and, for the initial value problem, the solution lies in
$C(\mathbb{R}^{+};\mathbb{R}^{d})$, the space of continuous
functions from $\mathbb{R}^+$ to $\mathbb{R}^d$. A consequence of
a series of results (Theorems 1--3) in ~\cite{CarrD2} is the
following: when all the eigenvalues of $A$ have negative real
parts, there exists $\alpha\in\mathbb{R}$ such that
\[
\|x(t)\| = \mathcal{O}(t^{\alpha}), \quad \mbox{as
$t\rightarrow\infty$}.
\]
We show that a similar result holds for the equation
(\ref{eq:eq73}), under the hypothesis
\begin{equation}     \label{eq:eq7H2}
\mbox{All eigenvalues of $A$ have negative real parts}.
\end{equation}
If the matrix $A$ satisfies (\ref{eq:eq7H2}), there exists a
positive definite matrix $C$ such that
\begin{equation}     \label{eq:eq731}
A^{T}C+CA=-I,
\end{equation}
where $A^{T}$ denotes the transpose of the matrix $A$, and $I$ is
the $d\times d$ identity matrix. Further, let
$\underline{\gamma}^2$ and $\overline{\gamma}^2$, with
$0<\underline{\gamma}^2\leq\overline{\gamma}^2$, be the minimal
and maximal eigenvalues of $C$. Define $V(x)=\langle x,Cx\rangle$,
for all $x\in\mathbb{R}^{d}$. By construction
$\underline{\gamma}^2\|x\|^2\leq V(x)\leq
\overline{\gamma}^2\|x\|^2$. Using the Cauchy-Schwarz inequality
and the inequality $2uv\leq \eta_1^2u^2+\eta_1^{-2}v^2$ for $u$,
$v\in\mathbb{R}$, we see for any $x$, $y\in \mathbb{R}^{d}$ that
\begin{equation}     \label{eq:eq732}
|\langle x,CBy\rangle + \langle By,Cx\rangle| \leq
\|B^{T}C\|\left(\eta_1^2\;\frac{1}{\underline{\gamma}^2}\;V(x)
+\frac{1}{\eta_1^2}\;\frac{1}{\underline{\gamma}^2}\;V(y) \right).
\end{equation}
Moreover, as
\[
\langle C(\Sigma x+\Theta y),\Sigma x+ \Theta y \rangle= \langle
x,\Sigma^TC\Sigma x\rangle+ \langle y,\Theta^TC\Theta y\rangle+
2\langle x,\Sigma^T C\Theta y\rangle,
\]
we can use $2uv\leq \eta_2^2u^2+\eta_2^{-2}v^2$ for $u$,
$v\in\mathbb{R}$, to obtain
\begin{multline}     \label{eq:eq733}
\langle C(\Sigma x+\Theta y),\Sigma x+ \Theta y \rangle
\leq\frac{1}{\underline{\gamma}^2}\;\|\Sigma^TC\Sigma\|\;V(x)
+\frac{1}{\underline{\gamma}^2}\;\|\Theta^TC\Theta\|\;V(y)
\\+ \|\Sigma^T C\Theta\|
\left( \eta_2^2\;\frac{1}{\underline{\gamma}^2}\;V(x)
+\frac{1}{\eta_2^2}\;\frac{1}{\underline{\gamma}^2}\;V(y) \right).
\end{multline}
Finally, we have
\begin{equation}    \label{eq:eq734}
-\langle x,x\rangle \leq \frac{1}{\overline{\gamma}^2}\;V(x).
\end{equation}
Using It\^{o}'s rule and (\ref{eq:eq731}), we have
\begin{multline*}
V(X(t))=V(X(0)) +\int_{0}^{t} -\langle X(s),X(s)\rangle
+\langle X(s),CBX(qs)\rangle \\
+ \langle BX(qs),CX(s)\rangle +\langle C(\Sigma X(s)+ \Theta
X(qs)),\Sigma X(s)+\Theta X(qs)\rangle\,ds
\\
+\int_{0}^{t} \langle X(s),C(\Sigma X(s)+\Theta X(qs))\rangle +
\langle CX(s),\Sigma X(s)+ \Theta X(qs)\rangle\,dB(s).
\end{multline*}
Putting $m(t)=\mathbb{E}[V(X(t))]$, we proceed as before, using
(\ref{eq:eq732}), (\ref{eq:eq733}) and (\ref{eq:eq734}) to show,
for any $\varepsilon>0$, that $m(t)\leq y(t)$, where we have
$y(0)=(1+\varepsilon)m(0)>m(0)$, and
\begin{multline*}
y'(t)=\left( -\frac{1}{\overline{\gamma}^2}+\frac{\|B^T
C\|}{\underline{\gamma}^2}\;\eta_1^2
+\frac{1}{\underline{\gamma}^2}\; \|\Sigma^TC\Sigma\|+
\frac{1}{\underline{\gamma}^2}\; \|\Sigma^TC\Theta\|\;\eta_2^2
\right)y(t)
\\+
\left( \frac{\|B^TC\|}{\underline{\gamma}^2}\;\frac{1}{\eta_1^2}
+\frac{1}{\underline{\gamma}^2}\; \|\Theta^T C\Theta\|
+\frac{1}{\underline{\gamma}^2}\;\|\Sigma^T
C\Theta\|\;\frac{1}{\eta_2^2} \right) y(qt).
\end{multline*}
Notice moreover that we also have $m(0)=\mathbb{E}[\langle
X(0),CX(0)\rangle] \leq \|C\|\mathbb{E}[\|X_0\|^2]$. Thus
\[
\mathbb{E}[\|X(t)\|^2]\leq \frac{1}{\underline{\gamma}^2}\;m(t)
\leq\frac{1}{\underline{\gamma}^2}\;y(t).
\]
We are now in a
position to distill the result of this discussion into the
following theorem.
\begin{theorem}      \label{theorem:t72}
Let $(X(t))_{t \geq 0}$ be the solution of (\ref{eq:eq73}) with
$\mathbb{E}(\|X_0\|^4)< \infty$. Suppose that (\ref{eq:eq7H2})
holds. If $\underline{\gamma}^2$ and $\overline{\gamma}^2$, with
$0<\underline{\gamma}^2\leq \overline{\gamma}^2$, are the maximum
and minimum eigenvalues of the positive definite matrix $C$ which
satisfies $A^TC+CA=-I$, then the following holds:
\begin{itemize}
\item[(i)] Suppose
\[
\frac{1}{\overline{\gamma}^2} + \frac{1}{\underline{\gamma}^2}\;
\|\Sigma^TC\Sigma\|<0.
\]
Then there exist $\alpha\in\mathbb{R}$ and $C>0$ such that
\begin{equation}    \label{eq:eq725}
\mathbb{E}(\|X(t)\|^2)\leq C\;\mathbb{E}[\|X(0)\|^2]\;t^{\alpha},
\end{equation}
and
\begin{equation}     \label{eq:eq727}
\limsup_{t\rightarrow \infty} \frac{\log\|X(t)\|}{\log t}
\leq\frac{1}{2}(\alpha+1), \quad \mbox{a.s.}
\end{equation}
Therefore, the process is globally polynomially bounded in
mean-square and is almost surely globally polynomially bounded.
\item[(ii)] Suppose
\[
-\frac{1}{\overline{\gamma}^2}+\frac{1}{\underline{\gamma}^2}
\left(\|\Sigma^TC\Sigma\|+\|\Theta^TC\Theta\|+2\|B^TC\|
+2\|\Sigma^TC\Theta\|\right)<0.
\]
Then there exists $\alpha<0$ such that (\ref{eq:eq725}) and
(\ref{eq:eq727}) hold. Therefore the equilibrium solution of
(\ref{eq:eq73}) is globally polynomially stable in mean-square,
and the solution process of (\ref{eq:eq73}) is almost surely
globally polynomially bounded.
\item[(iii)] Suppose
\[
-\frac{1}{\overline{\gamma}^2}+\frac{1}{\underline{\gamma}^2}
\left(\|\Sigma^TC\Sigma\|+\frac{\|\Theta^TC\Theta\|}{q}
+\frac{2}{\sqrt{q}}\;\|B^TC\|
+\frac{2}{\sqrt{q}}\;\|\Sigma^TC\Theta\|\right)<0.
\]
Then there exists $\alpha<-1$ such that (\ref{eq:eq725}) and
(\ref{eq:eq727}) hold. Therefore the equilibrium solution of
(\ref{eq:eq73}) is globally polynomially stable in mean-square,
and a.s. globally polynomially stable.
\end{itemize}
\end{theorem}
We note that we can obtain weaker sufficient conditions under
which the conclusions of Theorem \ref{theorem:t72} hold, which can
be more easily interpreted and verified. Instead of
(\ref{eq:eq732}) and (\ref{eq:eq733}), we have
\begin{equation*}
|\langle x,CBy\rangle + \langle By,Cx\rangle| \leq
\|B\|\;\|C\|\left(\eta_1^2\;\frac{1}{\underline{\gamma}^2}\;V(x)
+\frac{1}{\eta_1^2}\;\frac{1}{\underline{\gamma}^2}\;V(y) \right)
\end{equation*}
and
\begin{multline*}
\langle C(\Sigma x+\Theta y),\Sigma x+ \Theta y \rangle
\leq\frac{1}{\underline{\gamma}^2}\;\|\Sigma\|^2\;\|C\|\;V(x)
+\frac{1}{\underline{\gamma}^2}\;\|\Theta\|^2\;\|C\|\;V(y)
\\+ \|\Sigma\|\;\|C\|\;\|\Theta\|
\left( \eta_2^2\;\frac{1}{\underline{\gamma}^2}\;V(x)
+\frac{1}{\eta_2^2}\;\frac{1}{\underline{\gamma}^2}\;V(y) \right).
\end{multline*}
Since $C$ is symmetric, $\|C\|=\overline{\gamma}^2$. Using these
estimates in the proof of Theorem \ref{theorem:t72} gives the
following result.
\begin{corollary}   \label{corollary:c71}
Let $(X(t))_{t \geq 0}$ be the solution of (\ref{eq:eq73}) with
$\mathbb{E}(\|X_0\|^4)< \infty$. Suppose that (\ref{eq:eq7H2})
holds. If $\underline{\gamma}^2$ and $ \overline{\gamma}^2$, with
$0<\underline{\gamma}^2\leq \overline{\gamma}^2$, are the maximum
and minimum eigenvalues of the positive definite matrix $C$ which
satisfies $A^TC+CA=-I$, then the following holds:
\begin{itemize}
\item[(i)] Suppose
\[
\frac{\underline{\gamma}^2}{\overline{\gamma}^4} >\|\Sigma\|^2.
\]
Then there exist $\alpha\in\mathbb{R}$, $C>0$ such that
(\ref{eq:eq725}) and (\ref{eq:eq727}) hold.
\item[(ii)] Suppose
\[
\frac{\underline{\gamma}^2}{\overline{\gamma}^4}
>(\|\Sigma\|+\|\Theta\|)^2.
\]
Then there exists $\alpha<0$ such that (\ref{eq:eq725}) and
(\ref{eq:eq727}) hold.
\item[(iii)] Suppose
\[
\frac{\underline{\gamma}^2}{\overline{\gamma}^4}
>\frac{2}{\sqrt{q}}\|B\|+(\|\Sigma\|+\frac{\|\Theta\|}{\sqrt{q}})^2.
\]
Then there exists $\alpha<-1$ such that (\ref{eq:eq725}) and
(\ref{eq:eq727}) hold.
\end{itemize}
\end{corollary}
Therefore, if the intensities of the noise terms are sufficiently
small, the negative spectrum of $A$ ensures the polynomial
asymptotic stability of the noise perturbed system, thereby
following the global polynomial asymptotic stability exhibited for
(\ref{eq:eq716}). Note once again that the polynomial mean square
and almost sure bounds on the solution exist, provided the
intensity of noise from the non-delay term is sufficiently small
and $A$ has a negative spectrum: as is the case for equation
(\ref{eq:eq716}), the presence of polynomial asymptotic behaviour
does not seem to be determined by the terms with delayed
arguments.

\section{Concluding remarks}
There are a number of related problems which we have not studied
in this paper, which nonetheless can be treated using the analysis
presented here. For instance, all the stochastic evolutions we
have considered above are driven by a single Brownian motion, but
no new ideas are required to extend the results to stochastic
functional differential equations driven by finitely many Brownian
motions. Another interesting class of equations to study are the
stochastic analogues of the pantograph equations studied in, e.g.
\cite{MakTer, Kris}, where the delayed argument is not necessarily of
proportional form, and the rate of decay or growth of solutions is 
not necessarily polynomially fast. We hope to consider the asymptotic 
behaviour of such stochastic equations elsewhere. We have also 
omitted to study nonlinear and nonautonomous versions of the 
stochastic pantograph equation, for example
\[
dX(t) = \left(f_1(t,X(t))+g_1(t,X(qt))\right)\,dt +
\left(f_2(t,X(t))+g_2(t,X(qt))\right)\,dB(t),
\]
where $f_1$, $f_2$ are globally linearly bounded and Lipschitz
continuous. If however, the function $f_1$ satisfies
\[
\langle f_1(t,x),x\rangle \leq -a\|x\|^2
\]
for all $x\in\mathbb{R}$, $t\geq 0$ and some $a>0$, we can again
recover the polynomial asymptotic behaviour exhibited by the
processes studied in this paper for this process. Finite
dimensional and many--delay analogues of the extensions mentioned
here can be treated using the techniques of the previous section.

\section{Appendix} We will require some further supporting results from
stochastic analysis, and elementary properties of the stochastic
pantograph equation.

Before we give the proof of Lemma~\ref{lemma:lemp1}, which was
earlier deferred, we first recall for a continuous real-valued
function $f$ of a real variable, the Dini-derivative
$D_{-}f$ is defined as
\[
D_{-} f(t) = \liminf_{\delta \uparrow 0}
\frac{f(t+\delta)-f(t)}{\delta}.
\]
To prove Lemma~\ref{lemma:lemp1}, we also require the following
result, which appears as Lakshmikantham and Leela~\cite[Vol. 1,
Lemma 1.2.2]{lakleelav1}.

\begin{lemma}    \label{lemma:llakleela}
Let $v,w$ be continuous functions and $D v(t) \leq w(t)$ for $t$
in an interval with a possible exceptional set of measure zero and
$D$ being a fixed Dini-derivative. Then $D_- v(t) \leq w(t)$ holds
for $t$ in an interval, with a possible exceptional set of measure
zero.
\end{lemma}
We now turn to the proof of Lemma~\ref{lemma:lemp1}.
\begin{proof}
With $x$ defined by (\ref{detpantoeps}) and $p$ defined by
(\ref{eq:eqg1}) with $0<p(0)<x(0)$, we will show that
\begin{equation}\label{toshow}
p(t) < x(t)
\end{equation}
for all $t \geq 0$. Assume that (\ref{toshow}) is false; then
there exists $T>0$ such that the set
\[
Z = \{ t \in [0,T) : p(t) \geq x(t) \}
\]
is nonempty. We set $t_1 = \inf Z$ and as $p(0)<x(0)$, we have
that $0 < t_1$ and
\[
p(t_1) = x(t_1) \quad \mbox{and} \quad p(t) < x(t) \quad
\mbox{for} \quad t \in [0,t_1).
\]
Then, for all $-|t_1|<h<0$, as $p(t_1+h)<x(t_1+h)$,
\[
\frac{1}{h} \{(p(t_1+h) - (p(t_1) \}
> \frac{1}{h} \{x(t_1+h)) - x(t_1)) \},
\]
so letting $h\uparrow0$, we get
\begin{equation} \label{eq:eqfact1}
D_- p(t_1)\geq x'(t_1).
\end{equation}
Next by Lemma~\ref{lemma:llakleela}, as $p$ obeys (\ref{eq:eqg1})
and $t\mapsto  {\bar a}p(t) + {\bar b} p(qt)$ is continuous, we
have
\begin{equation*}
D_- p(t) \leq {\bar a} p(t) + {\bar b} p(qt), \quad t>0.
\end{equation*}
In particular, this yields
\begin{equation}   \label{eq:eqfact2}
D_- p(t_1) \leq {\bar a} p(t_1) + {\bar b} p(qt_1).
\end{equation}
Thus, (\ref{eq:eqfact1}), (\ref{eq:eqfact2}) and
(\ref{detpantoeps}), together with the facts that $p(t_1)=x(t_1)$,
and $p(qt_1)<x(qt_1)$, imply
\begin{eqnarray*}
0&\leq& D_- p(t_1) - x'(t_1) \\
&\leq& {\bar a} p(t_1) + {\bar b} p(qt_1)
          - ({\bar a} x(t_1) \ + \ {\bar b} x(qt_1)) \\
&=& {\bar b}(p(qt_1)-x(qt_1)) \\
&<&0,
\end{eqnarray*}
which is a contradiction. The result therefore follows.
\end{proof}

The following result is due to LeGall \cite{LeGall}.
\begin{lemma} \label{lemma:LeGall}
Suppose $X$ is a continuous semi-martingale with decomposition
\[
X(t) = X_0 + V(t) + M(t),
\]
where $M$ is a continuous local martingale and $V$ is the
difference of continuous, non-decreasing adapted processes with
$V_0 =0$, a.s. If $k: (0,\infty) \rightarrow (0,\infty)$ is a
Borel function satisfying
\[
\int_0^{\varepsilon} \frac{1}{k(x)} \ dx = \infty \qquad\mbox{for
all } \varepsilon >0,
\]
and
\[
\int_0^{t} \frac{d \langle M \rangle (s)}{k(X(s))} \ 1_{\{X(s)>
0\} } ds < \infty \qquad \mbox{for all } t \ \mbox{a.s.},
\]
then the semi-martingale local time of $X$ at $0$ is identically
$0$, almost surely.
\end{lemma}
\begin{remark} \label{remark:r3}
For the equation (\ref{stochpanto}) with $\rho=0$ i.e.,
\begin{equation}     \label{eq:eqg5}
dX(t) = (aX(t)+bX(qt))\,dt + \sigma X(t)\,dB(t),
\end{equation}
note that $k(x)=x^2$ satisfies the conditions of Lemma
\ref{lemma:LeGall}.
\end{remark}
In the next Lemma, we extend results on moment bounds for the
stochastic pantograph equation from \cite{BakerBuck2}.
\begin{lemma}        \label{lemma:lemg2}
If $X(t;0,X_0)$ is a solution of (\ref{stochpanto}) with
$\mathbb{E}(X_0^4)< \infty$, then
\begin{enumerate}
\item[(i)]
$t\mapsto \mathbb{E}(|X(t)|^m)$ is continuous on $\mathbb{R}^+$
for $m =1,2$,
\item[(ii)]
$\mathbb{E}\int_0^t |X(s)|^{2m}\,ds < \infty$ for all fixed $t\geq
0$ and for $m =1,2$.
\end{enumerate}
\end{lemma}
\begin{proof}
In \cite{BakerBuck2} it is shown, that $\mathbb{E}(|X(t) - X(s)|^2)
\leq C(T) \cdot(t-s)$ for any $0\leq s \leq t \leq T$ with $t-s <
1$. Part (i) for $m=1$ now follows from the inequalities:
\begin{eqnarray*}
\left|\mathbb{E}(|X(t)|) - \mathbb{E}(|X(s)|)\right| &\leq&
\mathbb{E}\left| |X(t)|-|X(s)| \right|
\leq \mathbb{E}(\left|X(t) -  X(s)\right|)\\
&\leq& \left(\mathbb{E}(X(t) -  X(s))^2\right)^{\frac{1}{2}}.
\end{eqnarray*}
For Part (i) for $m=2$ we have:
\begin{eqnarray*}
\Big| \mathbb{E}(X^2(t)) - \mathbb{E}(X^2(s)) \Big|
 &=& \Big|\mathbb{E}(X^2(t) - X^2(s))\Big| \\
 &\leq& \mathbb{E}[\left|X^2(t) -  X^2(s) \right|]\\
 &\leq& \mathbb{E}[\left|(X(t) -  X(s))(X(t) +  X(s)) \right|]\\
 &\leq& \mathbb{E}\left| X(t) -  X(s)\right|^2 \
\mathbb{E} \left|X(t) +  X(s)\right|^2.
\end{eqnarray*}
Now, $t \mapsto\mathbb{E}\Big(\sup_{0\leq s\leq t}
|X(s)|^{2}\Big)$ is bounded on compact sets, \cite[Theorem
2.3]{BakerBuck2}, so applying this proves (i) for $m=2$.

To establish (ii), note that, for $m=1$, it follows immediately
from the last statement. Modifying the argument of \cite[Theorem
2.3]{BakerBuck2} establishes that the function
$t\mapsto\mathbb{E}\Big(\sup_{0\leq s\leq t} |X(s)|^{4}\Big)$ is
bounded, which proves (ii) for $m=2$.
\end{proof}

\begin{flushleft}
\textbf{Acknowledgements}
\end{flushleft}
The authors were supported by a SQuaRE activity entitled ``Stochastic stabilisation of limit-cycle dynamics in ecology and neuroscience'' funded by the American Institute of Mathematics.

\end{document}